%% file: ArXiV-main.tex
\documentclass{amsart}
\usepackage[table,dvipsnames]{xcolor}
\newcommand{\Def}[1]
{\textcolor{NiceBlue!90}{\textit{#1}}}
\usepackage[
	bookmarks=true,
	plainpages=false,
	linktocpage,
	colorlinks=true,
	citecolor=green!80!black,
	linkcolor=red!70!black,
	filecolor=magenta,
	urlcolor=magenta,
	breaklinks,
	pdfauthor={Martin Winter},
]{hyperref}

\newcommand{\real}{\realSp}
\newcommand{\realSp}{\ensuremath\text{\upshape\textsc{real}}}
\newcommand{\realCvx}{\ensuremath\text{\upshape\textsc{real}}_{\mathrm c}}
\newcommand{\realCvxCl}{\ensuremath\overline{\text{\upshape\textsc{real}}}_{\mathrm c}}

\theoremstyle{theorem}
\newtheorem{innercustomgeneric}{\customgenericname}
\providecommand{\customgenericname}{}
\newcommand{\newcustomtheorem}[2]{%
  \newenvironment{#1}[1]
  {%
   \renewcommand\customgenericname{#2}%
   \renewcommand\theinnercustomgeneric{##1}%
   \innercustomgeneric
  }
  {\endinnercustomgeneric}
}
\newcustomtheorem{theoremX}{Theorem}
\newcustomtheorem{lemmaX}{Lemma}
\newcustomtheorem{conjectureX}{Conjecture}

\newcommand{\Y}{\ensuremath{\mathrm{Y}}}

\newcommand{\DY}{\ensuremath{\Delta\kern-1.2pt \Y}}
\newcommand{\YD}{\ensuremath{\Y\kern-1.2pt\Delta}}

\input{header}
\usepackage{enumitem,moreenum}
\newcommand{\labelstyle}[1]{\upshape(\textit{#1})}
\newcommand{\mylabel}{\labelstyle{\roman*}}
\newenvironment{myenumerate}{\begin{enumerate}[label=\mylabel]}{\end{enumerate}}
\newenvironment{newmyenumerate}[1][]{\begin{enumerate}[#1]}{\end{enumerate}}

\newenvironment{myitemize}[1][]{\begin{itemize}[#1]}{\end{itemize}}

\def\itm#1{{\labelstyle{\romannumeral#1\relax}}}

\numberwithin{equation}{section}

\newtheoremstyle{mythmstyle} 
    {\parsep}                    
    {\parsep}                    
    {\itshape}                   
    {}                           
    {\bfseries\scshape}          
    {.}                          
    {.5em}                       
    {}  
\newtheoremstyle{mydefstyle} 
    {\parsep}                    
    {\parsep}                    
    {}                   
    {}                           
    {\mdseries\scshape}          
    {.}                          
    {.5em}                       
    {}  
\theoremstyle{theorem}
\newtheorem{theorem}{Theorem}[section]
\newtheorem{proposition}[theorem]{Proposition}

\newtheorem{lemma}[theorem]{Lemma}
\newtheorem{conjecture}[theorem]{Conjecture}

\theoremstyle{definition}
\newtheorem{definition}[theorem]{Definition}
\newtheorem{example}[theorem]{Example}
\newtheorem{remark}[theorem]{Remark}
\newtheorem{question}[theorem]{Question}

\crefname{theorem}{Theorem}{Theorems}
\crefname{proposition}{Proposition}{Propositions}
\crefname{lemma}{Lemma}{Lemmas}
\crefname{corollary}{Corollary}{Corollaries}
\crefname{remark}{Remark}{Remarks}
\crefname{example}{Example}{Examples}
\crefname{definition}{Definition}{Definitions}
\crefname{problem}{Problem}{Problems}
\crefname{observation}{Observation}{Observation}
\crefname{construction}{Construction}{Construction}

\begin{document}

\title[Rigidity of Polytopes]{Rigidity of polytopes with edge length and coplanarity constraints}

\author[M.\ Himmelmann]{Matthias Himmelmann}
\address{Institute of Analysis and Algebra, Technische Universit\"at Braunschweig, Germany}
\email{matthias.himmelmann@tu-braunschweig.de}

\author[B.\ Schulze]{Bernd Schulze}
\address{School of Mathematical Sciences, Lancaster University, Lancaster,
LA1 4YF, UK}
\email{b.schulze@lancaster.ac.uk}

\author[M.\ Winter]{Martin Winter}
\address{Nonlinear Algebra Group, Max-Planck-Institute for Mathematics in the Sciences, Leipzig, Germany}
\email{martin.winter@mis.mpg.de}

%

		
\date{\today}
\begin{abstract}
We investigate a novel setting for polytope rigidity, where a flex must preserve edge lengths and the planarity of faces, but is allowed to change the shapes of faces. For instance, the regular cube is flexible in this notion.
We present techniques for constructing flexible polytopes and find that flexibility seems to be an exceptional property.
Based on this observation, we introduce a notion of generic realizations for polytopes and conjecture that convex polytopes are generically rigid in dimension $d\geq 3$. We prove~this~con\-jecture in dimension $d=3$.
Motivated by our findings we also pose several questions that are intended to inspire future research into this notion of polytope rigidity. 
\end{abstract}

\newpage

\maketitle

\input{sec/introduction}

\par\bigskip
\noindent
\textbf{Funding.} 
Matthias Himmelmann is grateful for the financial support from the project ``Discretization in Geometry
and Dynamics'' (SFB 109) in the Deutsche Forschungsgemeinschaft
(DFG) and acknowledges the support by the National Science Foundation under Grant No.\ DMS-1929284 while the author was in residence at the Institute for Computational and Experimental Research in Mathematics in Providence, RI, during the Geometry of Materials, Packings and Rigid Frameworks semester program.

Martin Winter was supported as Dirichlet Fellow by the Berlin Mathematics Research Center MATH\raisebox{0.25ex}{$+$} and the Berlin Mathematical School, funded by the Deutsche Forschungsgemeinschaft (DFG, German Research Foundation) under Germany’s Excellence Strategy (EXC-2046/1, project ID 390685689), and furthermore by the SPP 2458 ``Combinatorial Synergies'' (project ID 539851419), funded by the Deutsche Forschungsgemeinschaft (DFG, German Research Foundation).

\par\bigskip
\noindent
\textbf{Acknowledgements.} 
We thank Albert Zhang for helpful discussions. We also thank the Fields Institute for Research in Mathematical Sciences for providing a productive work environment and financial support during the focus program on Geometric Constraint Systems from July 1 to August 31, 2023. Moreover, we thank the Institute for Computational and Experimental Research in Mathematics (ICERM) for their hospitality and financial support during the program on Geometry of Materials, Packings and Rigid Frameworks from January 29 to May 2, 2025.


\bibliographystyle{abbrv}
\bibliography{reference}
\addresseshere

\appendix
\include{sec/appendix}

\end{document}

%% file: header.tex
\usepackage{amsmath,amsthm}
\usepackage{calc, mathtools} 
\usepackage{makecell}
\iftrue
\usepackage{amssymb} 
\else
\usepackage[charter,cal=cmcal]{mathdesign}
\fi

\usepackage{stmaryrd} 

\usepackage{colortbl,color} 
\usepackage{xcolor}
\definecolor{NiceBlue}{rgb}{0.15, 0.2, 0.75}
\usepackage{centernot} 
\usepackage{array} 

\usepackage[nameinlink,capitalize,noabbrev]{cleveref}
\usepackage{nicefrac}

\usepackage{tikz-cd} 

\usepackage[font=small,labelfont=bf]{caption}
\usepackage{subcaption}

\usepackage{blkarray}

\usepackage{inconsolata}

\usepackage[
    colorinlistoftodos,
    backgroundcolor=yellow,
    textsize=footnotesize 
]{todonotes}

\usepackage{pifont}

\usepackage{accents}
\newlength{\dhatheight}

\usepackage{bm}
\makeatletter
\newcommand{\addresseshere}{%
  \enddoc@text\let\enddoc@text\relax
}
\makeatother

\newcommand{\NN}{\mathbb{N}}    
\newcommand{\RR}{\mathbb{R}}    
                   
\newcommand{\bs}[1]{\boldsymbol{#1}}
\newcommand{\bsdot}[1]{\dot{\boldsymbol{#1}}}
\newcommand{\dotbs}[1]{\dot{\boldsymbol{#1}}}
\newcommand{\barbs}[1]{\bar{\boldsymbol{#1}}}
\newcommand{\tildebs}[1]{\tilde{\boldsymbol{{#1}}}}
\newcommand{\tilbs}[1]{\tildebs{#1}}
\newcommand{\hatbs}[1]{\hat{\boldsymbol{#1}}}

\newcommand{\hati}{{\ensuremath{\hat\imath}}}
\newcommand{\hatj}{{\ensuremath{\hat\jmath}}}

\renewcommand{\ij}{{ij}}
\newcommand{\hatihatj}{{\hati\hatj}}

\newcommand{\hatij}{{\ensuremath{\kern1pt\widehat{\kern-1pt\imath\kern-2.3pt\jmath}}}}

\makeatletter
\newcommand{\subalign}[1]{%
  \vcenter{%
    \Let@ \restore@math@cr \default@tag
    \baselineskip\fontdimen10 \scriptfont\tw@
    \advance\baselineskip\fontdimen12 \scriptfont\tw@
    \lineskip\thr@@\fontdimen8 \scriptfont\thr@@
    \lineskiplimit\lineskip
    \ialign{\hfil$\m@th\scriptstyle##$&$\m@th\scriptstyle{}##$\hfil\crcr
      #1\crcr
    }%
  }%
}
\makeatother

\newcommand{\contradiction}{\text{\Large$\smash{\lightning}$}}

\def\:{\colon}

\def\nlspace{\nolinebreak\space}
\def\nls{\nlspace}

\newcommand{\freespace}{\kern.07em}

\newcommand{\ul}[1]{\underline{\smash{#1}}}

\DeclareMathOperator{\aff}{aff}

\DeclareMathOperator{\conv}{conv}

\DeclareMathOperator{\GL}{GL}

\DeclareMathOperator{\corank}{corank}

\let\eps=\epsilon

\let\<=\langle
\let\>=\rangle

\def\...{...}
\newcommand{\shortStyle}{\textit}
\newcommand{\shortStylenf}{\normalfont}

\newcommand{\ie}{\shortStyle{i.e.,}}

\newcommand{\eg}{\shortStyle{e.g.}}

\newcommand{\wrt}{\shortStyle{w.r.t.}}

\newcommand{\cf}{\shortStylenf{cf.}}

\makeatletter
\renewcommand*{\eqref}[1]{%
  \hyperref[{#1}]{\textup{\tagform@{\ref*{#1}}}}%
}
\makeatother



%% file: sec/introduction.tex
\section{Introduction}
\label{sec:introduction}

We study the rigidity theory of point-hyperplane frameworks obtained from convex polytopes.
That is, we ask which convex polytopes $P\subset\RR^d$ permit continuous deformations that preserve the combinatorial type, edge lengths and the planarity of faces; but not necessarily the shapes of the faces.
Examples for $d = 3$ are shown in \cref{fig:flexes}.

\begin{figure}[h]
    \centering
    \begin{subfigure}[b]{0.9\textwidth}
        \centering
        \includegraphics[width=0.63\textwidth]{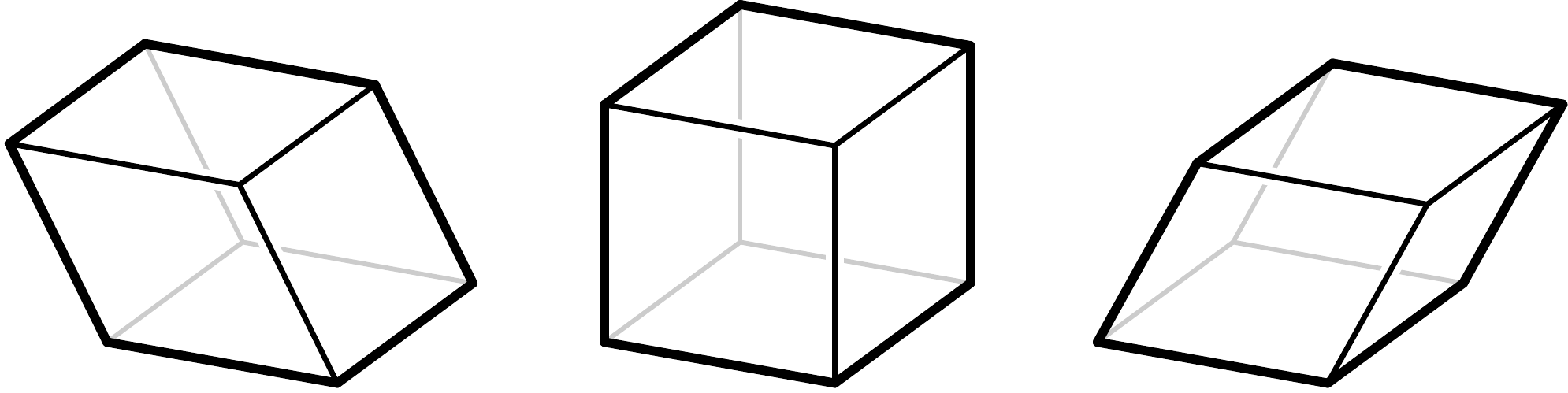}
        \caption{Continuous flex of a cube.}
        \label{fig:cube_flex}
    \end{subfigure}
    \begin{subfigure}[b]{0.52\textwidth}
        \centering
        \includegraphics[width=0.9\textwidth]{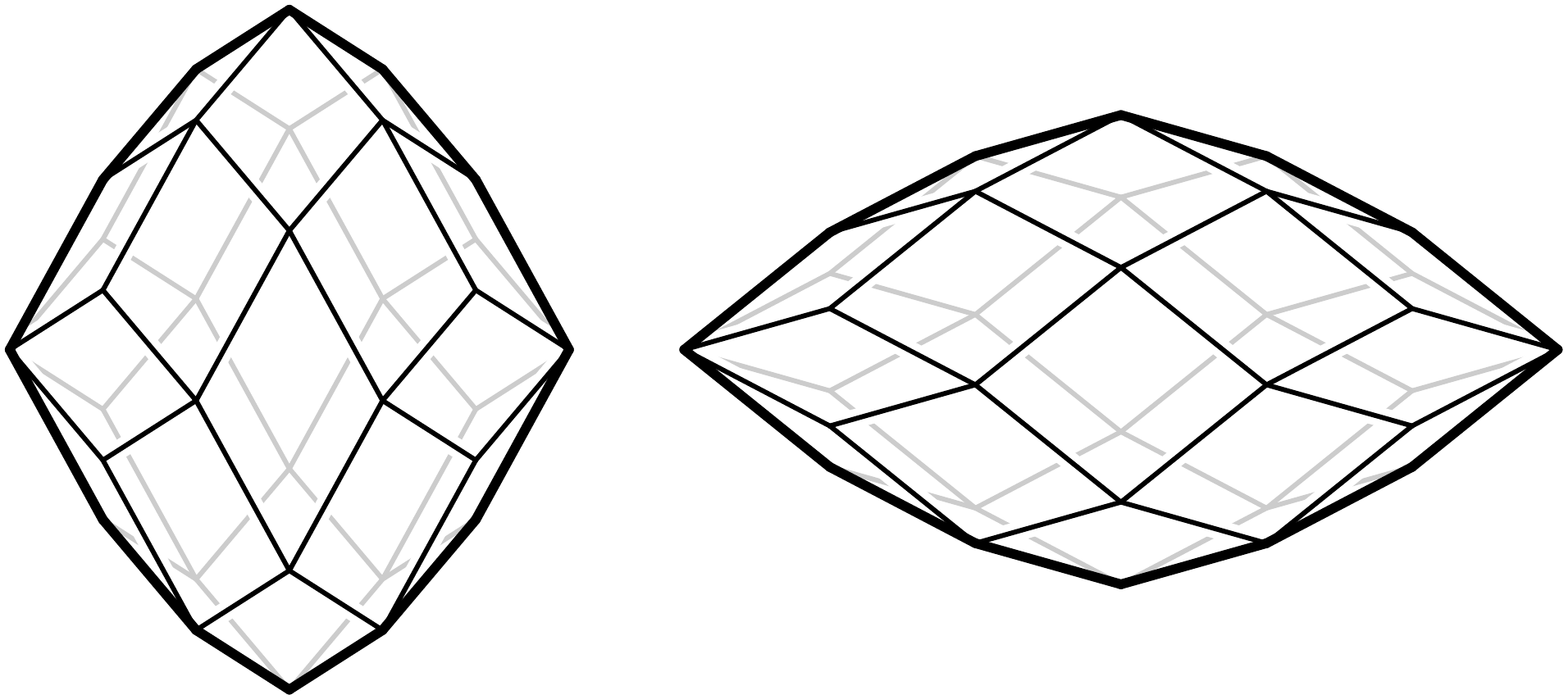}
        \caption{Continuous flex of a zonotope.}
        \label{fig:zonotope_flex}
    \end{subfigure}
    \begin{subfigure}[b]{0.47\textwidth}
        \centering
        \includegraphics[width=0.75\textwidth]{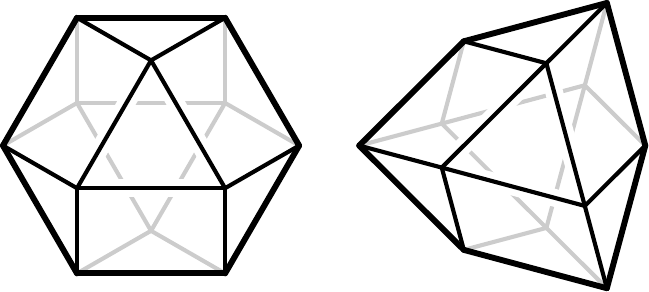}
        \caption{Continuous flex of the cuboctahedron.}
        \label{fig:cuboctahedron_flex_1}
    \end{subfigure}
    \caption{Examples of continuously flexing 3-dimensional polytopes.}
    \label{fig:flexes}
\end{figure}

Initial observations suggest that flexible polytopes (that is, polytopes that per\-mit a deformation in this sense) are rare, not only among combinatorial types,\nls but also~among realizations of a fixed combinatorial type.
This is not entirely unexpec\-ted: 
the Legendre-Steinitz theorem states that a 3-dimensional polytope has as~many edges as there are degrees of freedom in its realization space (see \eg\ \cite[Corollary 4.10]{rastanawi2021dimensions}).
This suggests that imposing one length constraint per edge should indeed~be sufficient to yield rigidity.
In higher dimensions one could argue that polytopes~are even overconstrained \cite[Remark 27.4]{pak2010lectures}.

While the existence of flexible polytopes shows that this heuristic is too simplistic, we still believe that it applies ``almost always''.
We will make this~precise using a notion of \emph{generic rigidity} for polytopes and propose the following conjecture:

\begin{conjecture}
    \label{conj:generic}
    Polytopes of dimension $d\ge 3$ are generically rigid.   
\end{conjecture}
A precise formulation, with all terms defined, requires some preparation and will be presented in \cref{sec:generic}.
The reader unfamiliar with ``generic rigidity'' should~read this statement as
\begin{quote}
    \itshape
    A randomly chosen realization of a polytope is rigid.
\end{quote}
The main result of this article is the proof of \cref{conj:generic} in dimension $d=3$: 
\begin{theorem}
    \label{res:generic}
    \label{thm:generic}
    Polytopes of dimension $d=3$ 
    are generically rigid.   
\end{theorem}

We emphasize that a generic realization does \emph{not} assume convexity.
Instead, it suffices to assume \emph{Zariski convex} -- a weaker condition than standard convexity,\nls and potentially quite general in dimension three (see \cref{def:generically_rigid}).

\begin{remark}
    \label{res:comparison}

\Cref{res:generic} combines features of several classical rigidity results~for polytopes
(for further details on the~classical results we refer to \cite[Chapter 26+]{pak2010lectures}):
\begin{myitemize}[leftmargin=2em,itemsep=0.75ex]
    \item 
    \Def{Cauchy's rigidity theorem} states that a convex polytope is determined by the shapes of its faces. Note that this not only asserts local rigidity, but~global~rigidity within the class of convex polytopes. In contrast to \cref{res:generic}~it~assumes that the shapes of faces do not change.
    \item 
    \Def{Dehn's theorem} asserts that any simplicial convex polytope is first-order rigid.
    An extension by Alexandrov (and later Whiteley) generalizes this statement~to non-simplicial polytopes where non-simplicial faces are triangulated without introducing new vertices in the interior, but potentially in the natural edges of the polytope~\cite{alexandrov2005convex,whiteley1984infinitesimally}.
    Even when allowing new vertices in the interior of faces, the resulting frameworks are second-order rigid (in fact, prestress~stable) due to Connelly and Gortler \cite{connelly1980rigidity,connelly2017prestress}.
    \item 
    \Def{Gluck's theorem} states that simplicial polytopes (not necessarily convex) are generically rigid \cite{gluck1975almost}. 
    In contrast to earlier results, which focus on the convex setting, non-convex polytopes can in fact be flexible, marking genericity~as an essential assumption.
    Famous examples of flexible polytopes are Bricard's~octahedra \cite{bricard1897memoire} as well as Connelly's \cite{connelly1977counterexample} and Steffen's \cite{steffen1978symmetric,lijingjiao2015asimoptimizing} flexible spheres. 
    A~new flexible polyhedron without self-intersections and attaining the theoretical minimum of eight vertices was recently found by Gallet, Grasegger, Legerský and Schicho \cite{gallet2024pentagonal}.
\end{myitemize}

An overview of the assumptions and conclusions of these results including \cref{res:generic} is given in \cref{tab:comparison}.
This comparison also reveals that \cref{res:generic} is best understood as an extension of Gluck's theorem to general combinatorial types.

\definecolor{ForestGreen}{RGB}{34,139,34}
\definecolor{Goldenrod}{RGB}{218,165,32}
\definecolor{Gray}{RGB}{95,95,95}
\newcommand{\goodcell}{\cellcolor{ForestGreen!25}\raisebox{-1pt}{No}}
\newcommand{\goodcellY}{\cellcolor{ForestGreen!25}\raisebox{-1pt}{Yes}}
\newcommand{\badcell}{\cellcolor{red!25}\raisebox{-1pt}{Yes}}
\newcommand{\okcell}{\cellcolor{Goldenrod!25}(\raisebox{-1pt}{Yes})}
\newcommand{\maybecell}{\cellcolor{Gray!25}{\raisebox{-1pt}?}}

\begin{table}[h!]
    \centering
    \begin{tabular}{|l|c|c|c|c|c|c|}
        \hline
        & assumes & assumes & assumes & assumes & extends &
        \\
        & rigid faces & convex & simplicial & generic & to $d>3$ & \makecell{proves}
        \\ \hline
        Cauchy & \badcell & \badcell & \goodcell & \goodcell & \goodcellY & global cvx
        \\ \hline
        Dehn & \okcell & \badcell & \badcell & \goodcell & \goodcellY & first-order
        \\ \hline
        Alex/Whit & \okcell & \badcell & \goodcell & \goodcell & \goodcellY & first-order
        \\ \hline
        Conn/Gor & \okcell & \badcell & \goodcell & \goodcell & \maybecell & second-order
        \\ \hline
        Gluck & \okcell & \goodcell & \badcell & \badcell & \goodcellY & first-order 
        \\ \hline
        \textbf{ours} & \goodcell & \goodcell\rlap & \goodcell & \badcell & \maybecell & first-order 
        \\ \hline           
    \end{tabular}
    \caption{
        Comparison of assumptions and conclusions for classical~rigi\-dity results and \cref{res:generic}.
        For Dehn and Gluck the rigidity of faces is implicit since the polytopes are simplicial.
        For Alexandrov/Whiteley and Connelly/Gortler the faces might change their shape only in so far as they are allowed to fold at triangulation edges. 
    }
    \label{tab:comparison}
\end{table}
\end{remark}

The rigidity and flexibility
analysis of polytopes with fixed edge lengths and planar, yet deformable, faces has potential
applications in engineering and robotics. In particular, it may inform the design of deployable and adaptable structures whose
faces are made from stretchable materials or membranes. The structures' geometrically-controlled deformations enable targeted functionality design
such as compact storage and dynamic shape-shifting \cite{MENG2023polyhedralmechanisms,zhai2018deployable}. Similar principles may also apply to biological systems, including the self-assembly of viruses, where capsids often arrange in polyhedral shapes \cite{Parvez2020geometricvirus,Perlmutter2015viralselfassembly}. Finally, these principles extend naturally to the design of quad meshes and origami-based structures, which offer compelling possibilities for use in architecture \cite{izmestiev2024tsurfaces,Sharifmoghaddam2024_thedra}.
Given that most polytopes are rigid, the design of any of the previously mentioned structures relies on a good understanding of the conditions under which flexibility can occur.

\subsection{Structure of the paper}

In \cref{sec:notation} we recall the necessary notions from polytopal combinatorics and geometry. 
Our notion of rigidity is formally introduced in \cref{sec:rigidity}, including the corresponding first-order theory.
In \cref{sec:examples}~we~provide the few examples of flexible polytopes that we are aware of and pose question derived from their structure.

\Cref{sec:generic} defines ``generic rigidity'' for polytopes and proves \cref{res:generic}.
The proof will span several subsections and makes use of a number of tools from the theory of planar graphs and polyhedral geometry, including the Maxwell-Cremona correspondence and Tutte embeddings.

\cref{sec:outlook} contains brief discussions of further aspects of polytope rigidity, such as higher dimensions, second-order rigidity and edge-length perturbations.

The paper also contains an appendix (\cref{sec:MCC_Tutte_topology_changes}) in which we prove that~both the Maxwell-Cremona correspondence and Tutte embeddings behave well under~certain limit operations. 
This result is not unexpected and is applicable in more general settings; still, we were not able to locate it in the literature.
The proof is not hard, but technical. 
Since a detailed treatment would distract from the main content, we decided to move it to the appendix. 

\section{Polytopes, combinatorial types and realizations}
\label{sec:notation}

Throughout 
this article let $P\subset\RR^d$ denote a $d$-dimensional \textit{convex}
polytope with non-empty interior and \Def{vertices} $p_1,...,p_n$ (\ie\ $P$ is the~convex hull $\conv\{p_1,...,p_n\}$).%

For our purpose, the combinatorics of a polytope can be described using its~vertex-facet incidence structure.
Recall that \Def{facet} refers to a maximal proper face, \ie\ a face of dimension $d-1$.
A \Def{combinatorial type}
$\mathcal P=(V,F,\sim)$ is a triple consisting of an (abstract) \Def{vertex set} $V$, an (abstract) \Def{facet set} $F$ and a \Def{vertex-facet incidence relation} ${\sim}\subseteq V\times F$.\footnote{
Usually, the combinatorial type of a polytope is defined via its \Def{face lattice}, that is, the partially ordered set made from the polytope's faces of all dimensions ordered by inclusion. Since the face-lattice can be reconstructed from the vertex-facet incidences, and vice versa, these notions are equivalent.}
Each convex polytope gives rise to such a structure.

To later enforce edge length constraints we 
also need to access vertex adjacency information in $\mathcal P$: we say that vertices $i,j\in V$ are \Def{adjacent} in $\mathcal P$, and denote this by $i\sim j$, if there are facets $\sigma_1,...,\sigma_r\in F$ so that $i$ and $j$ are the only two vertices incident to all of them.
This defines a graph structure on $V$ which we call the \Def{edge graph} $G_{\mathcal P}=(V,E)$.
If $\mathcal P$ is the combinatorial type of a convex polytope $P$, this notion agrees with the usual edge graph of $P$.
To emphasize the presence of this graph structure we also denote a combinatorial type by $(V,E,F,\sim)$.

A $d$-dimensional 
\Def{polytopal realization} of $\mathcal P$ is a pair $(\bs p,\bs a)$ consisting of a \textit{vertex map} $\bs p\: V\to\RR^d$ and a \textit{facet normal map} $\bs a\: F\to\RR^d$, so that for each $\sigma\in F$~we have $\|a_\sigma\|=1$ and the points (or \Def{vertices}) $\{p_i\mid i\sim\sigma\}$ lie on a common affine~hyperplane (the \Def{facet hyperplane}) with normal vector $a_\sigma$ (the \Def{facet normal}).\nls This can be~formally expressed as
$$\<p_j-p_i,a_\sigma\>=0,\quad\text{whenever $i,j\sim \sigma$}.$$
We say that a realization is (\Def{strictly})\! \Def{convex} if additionally
$$\<p_j-p_i,a_\sigma\><0,\quad\text{whenever $i\sim \sigma,j\not\sim\sigma$}.$$
Observe that each convex polytope $P$ with combinatorial type $\mathcal P$ gives rise to~a~convex realization in this sense and that the convex realizations are precisely the realizations obtained from convex polytopes. 
We often identify a polytope with the corresponding realization~and use $P$ and $(\bs p,\bs a)$ interchangeably.

All realizations of $\mathcal P$ taken together form the \Def{realization space} of $\mathcal P$: 
\begin{align*}
\realSp(\mathcal P)
&:=\left\{\!\!
\begin{array}{l}
    \bs p\:V\to\RR^d
    \\
    \bs a\:F\to\RR^d
\end{array}
\Big\vert
    \begin{array}{rcl}
    \|a_\sigma\| &\!\!\!\!=\!\!\!\!\!& 1 \text{ for all $\sigma\in F$} \\
    \<p_j-p_i,a_\sigma\> &\!\!\!\!=\!\!\!\!\!& 0\text{ whenever $i,j\sim \sigma$}\,
    \end{array}
\!\right\},
\end{align*}
Analogously, the \Def{convex realization space} is defined as
\begin{align*}
\realSp_{\mathrm c}(\mathcal P)
&:=\left\{\!\!
\begin{array}{l}
    \bs p\:V\to\RR^d
    \\
    \bs a\:F\to\RR^d
\end{array}
\Bigg\vert
    \begin{array}{rcl}
    \|a_\sigma\| &\!\!\!\!=\!\!\!\!\!& 1 \text{ for all $\sigma\in F$} \\
    \<p_j-p_i,a_\sigma\> &\!\!\!\!=\!\!\!\!\!& 0\text{ whenever $i,j\sim \sigma$}\,\\
    \<p_j-p_i,a_\sigma\> &\!\!\!\!<\!\!\!\!\!& 0\text{ whenever $i\sim \sigma, s\not\sim\sigma$}\,
    \end{array}
\!\right\},
\end{align*}
Clearly $\real(\mathcal P)$ is a real algebraic set in $\RR^{dV}\oplus\RR^{dF}$\!, and $\realCvx(\mathcal P)$ is a real semi-algebraic set in $\RR^{dV}\oplus\RR^{dF}$.

In later sections we focus specifically on polytopes in dimension three, which we call \Def{polyhedra}. 
We refer to \cref{sec:polyhedra} for any specific notation and background.

\section{Polytope rigidity}
\label{sec:rigidity}

We examine continuous deformations of $P$ that preserve edge lengths, analogous to bar-joint framework rigidity. Unlike in the framework setting, these motions~must also preserve the polytope's structure, that is, facet planarity and the combinatorial type.

Formally, we say that two realizations $(\bs p,\bs a),(\tilbs p,\tilbs a)\in\real(\mathcal P)$ are 
\begin{align*}
    \text{\Def{equivalent} if } \|p_j-p_i\|&=\|\tilde p_j-\tilde p_i\| \text{ for each edge $ij\in E$},
    \\
    \text{\Def{congruent} if } \|p_j-p_i\|&=\|\tilde p_j-\tilde p_i\|\text{  for each pair $i,j\in V$}.
\end{align*}
A \Def{motion} of $(\bs p,\bs a)$ is a curve $(\bs p^t,\bs a^t)\:[0,1]\to\real(\mathcal P)$ so that $(\bs p^t,\bs a^t)$ is equivalent to $(\bs p,\bs a)$ for all $t\in[0,1]$;
the motion is \Def{trivial} if $(\bs p^t,\bs a^t)$ is congruent to $(\bs p,\bs a)$ for all $t\in[0,1]$.
A non-trivial motion is called a \Def{flex}.
If a realization $(\bs p,\bs a)$ has a flex, then we call it \Def{flexible}, and \Def{rigid} otherwise.

Several examples of flexes in polytopes are shown in \cref{fig:flexes}.
More examples and concrete constructions are provided in \cref{sec:examples}.
The central question in this setting is:

\begin{question}
    Can one characterize the class of polytopes that are rigid or flexible? 
\end{question}

The reader familiar with \emph{point-hyperplane frameworks} (that is, bar-joint frameworks with additional coplanarity constraints) recognizes our construction as~a~special instance thereof.
We refer to Eftekhari et al.\ \cite{EJNSTW} for a general introduction and also for further details on their first-order theory.
Our investigation can be seen as an attempt at understanding where the origin of such a framework from a polytope allows us to prove stronger results than in the more general setting.
In particular convexity and sphericity are properties known for having strong implications in many contexts.

\subsection{First-order theory}
\label{sec:first_order}

The first-order theory of polytope rigidity is the same as for point-hyperplane~frameworks.
We recall the essentials.

A \Def{first-order motion} $(\bsdot p,\bsdot a)$ of $P$ consisting of maps $\bsdot p\:V\to\RR^d$ and $\bsdot a\:F\to\RR^d$, 
that satisfy the following equations:
\begin{align}
    \<p_i-p_j,\dot p_i-\dot p_j\> &= 0 ,\quad\text{whenever $ij\in E$},
    \label{eq:flex_at_vertex}
    \\
    \<p_i-p_j,\dot a_\sigma\> + \<\dot p_i-\dot p_j, a_\sigma\> &= 0 ,\quad\text{whenever $i,j\sim\sigma$},
    \label{eq:flex_at_vertex_face}
    \\
    \<a_\sigma,\dot a_\sigma\> &= 0,\quad\text{whenever $\sigma\in F$}.
    \label{eq:flex_at_face}
\end{align}

A first-order motion is \Def{trivial} if in addition it satisfies
\begin{align}
    \<p_i-p_j,\dot p_i-\dot p_j\> &= 0 ,\quad\text{for any $i,j\in V$},
    \label{eq:trivial_motion_vertex}
    \\
    \<p_i-p_j,\dot a_\sigma\> + \<\dot p_i-\dot p_j, a_\sigma\> &= 0 ,\quad\text{for any $i,j\in V$ and $\sigma\in F$},
    \label{eq:trivial_motion_vertex_facet}
    \\
    \<a_\sigma,\dot a_\tau\> + \<\dot a_\sigma, a_\tau\> &= 0 ,\quad\text{for any $\sigma,\tau\in F$}.
    \label{eq:trivial_motion_facet}
\end{align}
A non-trivial first-order motion is called a \Def{first-order flex}.
A framework is \Def{first-order rigid} if it has no first-order flexes; it is \Def{first-order flexible} otherwise.
Many results from the bar-joint framework setting transfer in the expected way. In particular, if a polytope is first-order rigid, then it is rigid. 

The first-order flex conditions \eqref{eq:flex_at_vertex} -- \eqref{eq:flex_at_face} can be collected in a corresponding \Def{rigidity matrix} $\mathcal R(P)$ 
which satisfies $\mathcal R(P)(\dotbs p,\dotbs a)=0$ if and only if $(\dotbs p,\dotbs a)$ is a first-order motion.
Since we only consider polytopes with non-empty interior, they are in particular affinely spanning and
the rank condition of first-order rigidity applies: 

\begin{lemma}
    \label{res:rank_condition}
    We have
    $$\corank \mathcal R(P) \ge \binom{d+1}{2},$$
    with equality if and only if $P$ is first-order rigid.
\end{lemma}

\section{Flexible polytopes}
\label{sec:examples}

Since flexible polytopes seem rare, this section is dedicated to collecting the few constructions that we are aware of.
Clearly, convex polytopes of dimension $d=2$ (that is, polygons) are flexible if and only if they have more than three vertices~(see \cref{fig:flex_polygon}). We shall focus on dimension $d\ge 3$.

\begin{figure}[h!]
    \centering
    \includegraphics[width=0.68\linewidth]{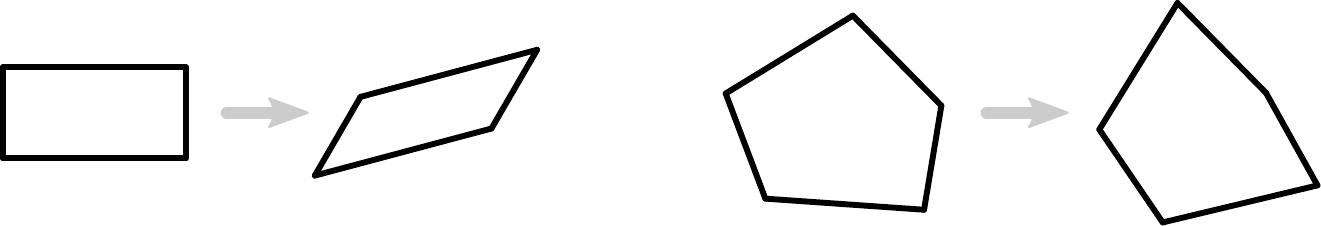}
    \caption{Flexible polygons.}
    \label{fig:flex_polygon}
\end{figure}

\subsection{Minkowski sums}
\label{sec:minkowski_sums}

Our most versatile tool for constructing flexible polytopes are Minkowski sums: for polytopes $P,Q\subset\RR^d$ the \Def{Minkowski sum} is
$$P+Q:=\{p+q\mid p\in P\text{ and }q\in Q\}.$$
Recall that if $P+Q$ has an edge $e$ in direction $v\in\RR^d$, then either $P$ or $Q$ (or both) have an edge in direction $v$.

Consider motions $P^t$ and $Q^t$ (which might be trivial) for which $P^t+Q^t$ stays of a fixed combinatorial type. 
For simplicity we assume that $P^t$ and $Q^t$ do not share edge directions. Then each edge of $P^t+Q^t$ corresponds to an edge of either $P^t$ or $Q^t$.
Since the motions $P^t$ and $Q^t$ preserve edge lengths, it follows that $P^t+Q^t$ is a motion as well.

There are generally two ways to obtain new flexible polytopes using this:
\begin{enumerate}
    \setlength{\itemsep}{1ex}
    \item 
    Even if $P^t$ and $Q^t$ are trivial motions, $P^t+Q^t$ can be a flex.
    For example, let $P^t:=P^0$ be the constant motion, and let $Q^t$ be a continuous reorientation.
    If the orientations of $P^0$ and $Q^0$ are generic (\ie\ have no parallel faces), then the combinatorial type of $P^t+{}$ $Q^t$ is preserved initially, which yields a flex (see \cref{fig:cuboctahedron_flex}).
\begin{figure}[h!]
    \centering
    \includegraphics[width=0.46\linewidth]{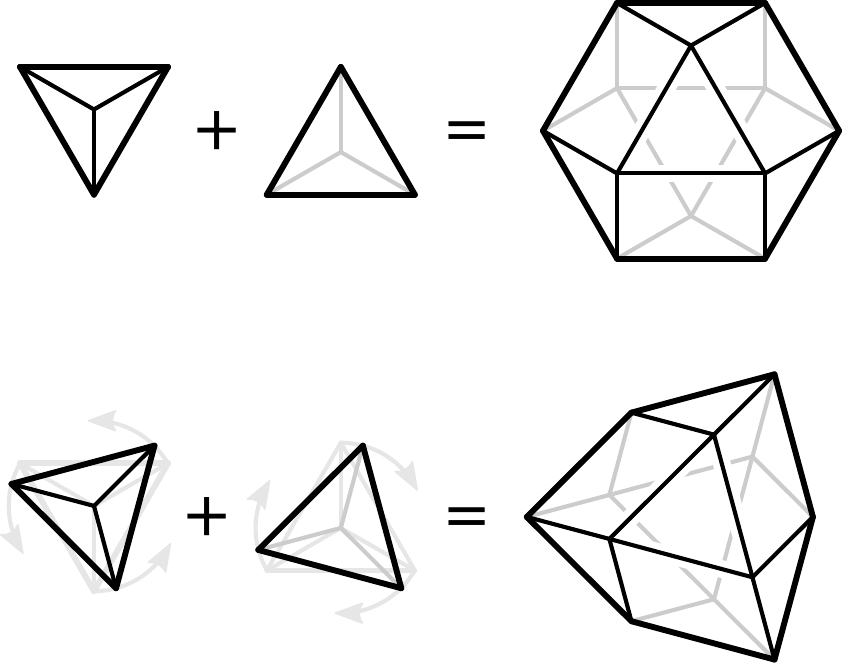}
    \caption{The cuboctahdron is the Minkowski sum of two simplices. Continuously changing the orientation of the summands yields a flex of the cuboctahedron.}
    \label{fig:cuboctahedron_flex}
\end{figure}
    \item 
    The Cartesian product $P^t \times Q^t$ is a special case of a Minkowski sum when the summands are in orthogonal subspaces.
    This enables the construction of new flexible polytopes using flexible polytopes from lower dimensions.
    In particular, we can use flexible polygons to obtain many flexible polytopes in dimensions $d\ge 3$ (such as prisms or duoprisms).
\end{enumerate}
We shall call flexes that are of the form $P^t+Q^t$ \Def{Minkowski flexes}.

\subsection{Zonotopes}
\label{sec:zonotopes}

A zonotope $Z=g_1+\cdots +g_r$ is the Minkowski sum of line~segments $g_i$. 
By the discussion of \cref{sec:minkowski_sums} we already know that a zonotope generated from generically oriented line segments is flexible.
We demonstrate here~that in fact \emph{every} zonotope is flexible (\eg\ the permutahedron).
The difficulty is that~an arbitrary reorientation of the $g_i$ does not necessarily preserve a zonotope's combinatorial structure; and it is not immediately clear that a combinatorics-preserving reorientation even exists.

For an arbitrary zonotope $Z=g_1+\cdots +g_r$ 
we choose a continuous family~$A^t\in\GL(\RR^d)$ of linear transformations that do not all preserve the angles between the $g_i$. 
Then,
$$g^t_i := \|g_i\|\cdot \frac{A^t g_i}{\|A^t g_i\|}$$ 
has the direction of $A^t g_i$ but the length of $g_i$. The zonotope $Z^t := g_1^t+\cdots + g_r^t$~has the same combinatorics and edge lengths as $Z$, but since $A^t$ does not preserve all angles between the $g_i$, $Z^t$ has different face angles than $Z$. We therefore defined a flex. 
An example of this is shown in \cref{fig:zonotope_flex} where $A^t$ are non-uniform scalings along an axis.

\subsection{Affine flexes and few edge directions}
\label{sec:affine_flexes}

An \Def{affine flex} is obtained through applying a continuous family of affine transformations (similar to the construction in \cref{sec:zonotopes}, but without rescaling edges afterwards).
It is known that a general framework has an affine flex if and only if its edge directions lie on a quadric~at~infinity (see \cite[Proposition 4.2]{connelly2005generic} or 
\cite[Proposition 1.4]{connelly2018affine} for details).
This happens~in particular if there are few edge directions. 
For example, in dimension $d=3$ each polytope with at most five edge direc\-tions has an affine flex because any five points in the plane lie on a quadric.
Examples are the cube and the polytope shown~in~\cref{fig:affine_flexible}.
The zonotope flex in \cref{fig:zonotope_flex} is an example of an affine flex with more than five edge directions (its edge directions lie on a circular cone, hence, ``on a~circle at infinity'').

\begin{figure}[h!]
    \centering    \includegraphics[width=0.21\linewidth]{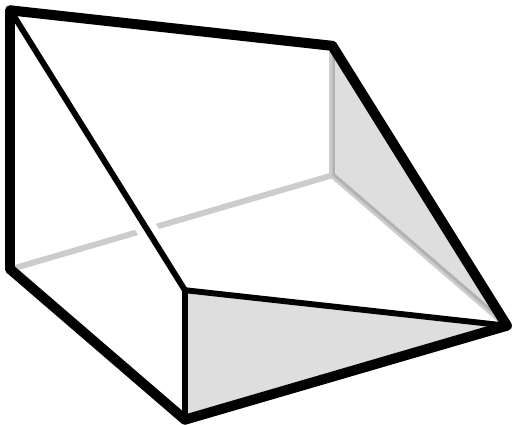}
    \caption{A 3-polytope that is affinely flexible because it has only five edge direction. It is also a Minkowski sum of the shaded faces.}
    \label{fig:affine_flexible}
\end{figure}

All examples of polytopes with affine flexes that we are aware of are Minkowski sums, and it is not clear whether this is always the case.

\begin{question}
    Is there an affine flex that is not a Minkowski flex?
\end{question}

\subsection{Stacking and irreducible flexes}

All concrete examples we encountered so far have been Minkowski sums. 
It is however not hard to come up with flexible~po\-lytopes that are not Minkowski sums, \eg\ by stacking pyramids onto some of their flexing faces (see \cref{fig:stacking}).

\begin{figure}[h!]
    \centering    \includegraphics[width=0.18\linewidth]{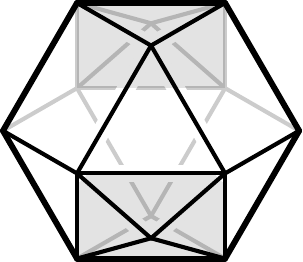}
    \caption{The cuboctahedron with two stacked faces is not~a~Minkow\-ski sum, but is still flexible. This is because the stacking restricts only one of the three degrees of freedom that come from the three relative rotations of the Minkowski summands of the cuboctahedron.}
    \label{fig:stacking}
\end{figure}

Even though the resulting polytopes are not Minkowski sums, the construction does not seem to yield a ``completely separate type of flexibility'' either. 

A polytope flex $P^t$ is said to be \Def{dissectable} if there is a family of hyperplanes~$E^t$ so that $P^t\cap H_{\pm}^t$ are motions of full-dimensional polytopes, where $H_+^t,H_-^t\subset\RR^d$ are the closed halfspaces defined by $E^t$ (see \cref{fig:slicing}).

\begin{figure}[h!]
    \centering
    \includegraphics[width=0.47\linewidth]{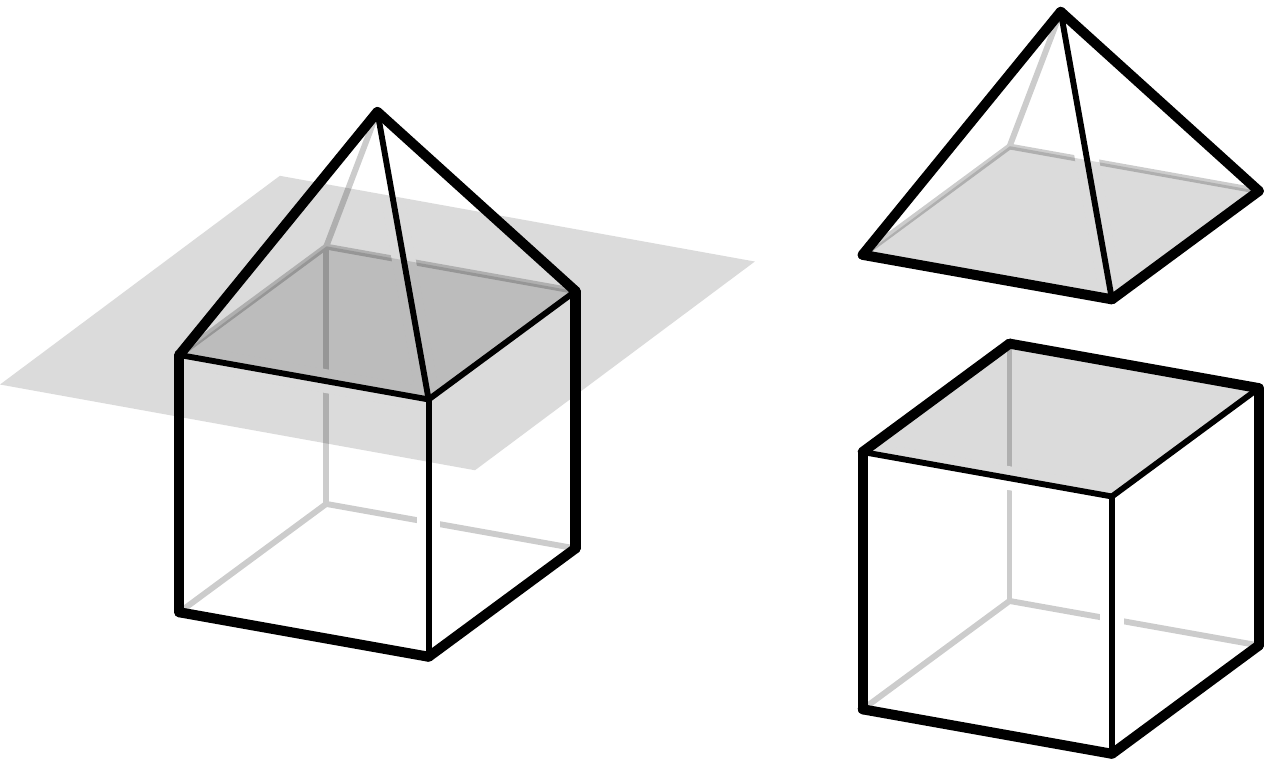}
    \caption{Each flex of the ``cube with pyramid roof'' is dissectable~since it can be split into a flexing cube and a rigid pyramid using the shown hyperplane.}
    \label{fig:slicing}
\end{figure}

\begin{question}
    Are there non-dissectable polytope flexes other than the Minkowski flex\-es?
\end{question}

\subsection{Further questions}

Based on the examples of flexible polytopes that we are aware of, a number of questions arise. 

\begin{question}
    \label{q:affine_trafo}
    Is polytope rigidity preserved under affine transformations?
\end{question}

This is in particular the case for Minkowski flexes and affine flexes.
Observe that already for Minkowski flexes it is not obvious how the flex changes under an affine transformation unless one knows a Minkowski decomposition.
One~should~note~that not even bar-joint rigidity, let alone point-hyper\-plane rigidity, is preserved under general affine transformations (see \eg\ the classifications of flexible frameworks of $K_{3,3}$ \cite{wunderlich1976deformable,husty2007nine}).

\begin{question}
    \label{q:parallel_edges}
    Does a flexible polytope with $d\ge 3$ necessarily have parallel edges?
\end{question}

Once again, this question is inspired by the examples that we know of and is our best guess for what a necessary structural property might look like.

\section{Generic rigidity of polytopes}
\label{sec:generic}
\label{sec:generic_rigidity}

In this section, we address \cref{conj:generic}, clarify the meaning of \Def{generic} in the context of polytopes, and prove our main result \cref{res:generic}.
We first provide a more precise formulation of the conjecture:

\begin{conjecture}
    \label{conj:generic_2}
    If $\mathcal P$ is the combinatorial type of a convex polytope of dimension $d\ge 3$ then $\mathcal P$ is generically (first-order) rigid.
\end{conjecture}

The precise definitions of ``generic'' and  ``generically rigid'' in our setting are~not obvious and we will discuss their subtleties in \cref{sec:generic_def} (culminating in \cref{def:generic,def:generically_rigid}).
A suitable definition of ``generic'' is expected to capture a notion of ``almost all realizations'', or of ``a randomly chosen realization'' of $\mathcal P$.
In analogy to other settings it is furthermore expected that a generic realization is rigid if and only if it is first-order rigid (hence the parentheses in the conjecture).

Our main result is the resolution of \cref{conj:generic} in the special case $d=3$:

\begin{theorem}
    \label{thm:generic_2}
    \label{res:generic_2}
    If $\mathcal P$ is the combinatorial type of a convex polyhedron (that is, $\mathcal P$~is~a polyhedral graph, see \cref{sec:polyhedra}), then $\mathcal P$ is generically (first-order) rigid.
\end{theorem}

Although \cref{conj:generic_2} remains open for $d\ge 4$, we will discuss potential approaches in \cref{sec:higher_dim}. 

\subsection{Overview of this section}

In \cref{sec:simplicial} we recall generic rigidity of simplicial polyhedra (in particular, Gluck's theorem). 
Based on this we work towards~definitions of ``generic'' and ``generically rigid'' for general combinatorial types in \cref{sec:generic_def} (\cf\ \cref{def:generic,def:generically_rigid}). 
In \cref{sec:polyhedra} we recall the necessary~prerequisites of polyhedral combinatorics and geometry. We shall see that our definition of ``generically rigid'' is especially natural for $d=3$.

The proof of \cref{res:generic} is presented starting from \cref{sec:proof_full}.
Some particularly technical parts of the proof are moved to \cref{sec:sequence,sec:limit_is_stressed_framework}.
Additional intermediate \cref{sec:Tutte_MC,sec:contraction_terminology} recall essential tools, such as Tutte embeddings and the Maxwell-Cremona correspondence.

\subsection{Simplicial polyhedra}
\label{sec:Gluck}
\label{sec:simplicial}


From the perspective of rigidity theory, simplicial~po\-lytopes are merely bar-joint frameworks. Hence, their notion of generic rigidity is directly inherited from the framework setting: a framework is said to be \emph{generic} (or \emph{regular}) if its rigidity matrix has maximal possible rank among all frameworks of the same graph (see for example \cite[p.\ 176]{asimow1979rigidity} or \cite[p.\ 13]{whiteley1996some}).
\textcolor{black}{This notion of ``generic'' is derived from the fact that the rigidity matrix of a framework is, up to a factor, precisely the Jacobian of the underlying polynomial system.
Rank deficiency of the Jacobian is the standard criterion for detecting singular points, away from which the configuration space behaves like a smooth manifold \cite[Section I.5]{hartshorne2013algebraic}.}
Since~rank deficiency~is also an algebraic condition, generic frameworks form a dense open~sub\-set of $(\RR^d)^V$.
A graph $G$ is \emph{generically rigid} if every generic framework of $G$ is~(first-order) rigid.

Once Dehn's theorem is available, a generic version (\ie\ Gluck's theorem) follows quickly (see \cref{res:comparison} to recall the statements of these theorems).

First, recall that a $d$-dimensional bar-joint framework (with at least $d+1$ vertices) is first-order rigid if and only if its rigidity matrix has corank $\binom{d+1}{2}$.
It~is~immediate from our definition of ``generic'' that then either all generic frameworks are first-order rigid, or all frameworks (including non-generic ones) are first-order flexible. 
Nevertheless, by Dehn's theorem we know that the latter cannot be the case for skeleta of simplicial polytopes: all \textit{convex} realizations of $\mathcal P$ (which always exist) are first-order rigid.
We conclude that $\mathcal P$ is generically rigid.

\subsection{Generic rigidity for general polytopes}
\label{sec:generic_def}

The discussion of the simplicial case suggests an analogous definition for generic realizations of a general combinatorial type $\mathcal P$: the realizations that have 
maximal rank among all realizations of $\mathcal P$, where the \Def{rank of a realization} refers to the rank of its associated rigidity matrix.
%
We do however encounter several problems with this straightforward definition.

In contrast to the simplicial case, realization spaces of general polytopes can~be arbitrarily complicated, already for convex polytopes and in dimensions as low as $d=4$ \cite{richter2006realization}.
Most relevant here, the (convex) realization space might be \emph{reducible} as a (semi-)algebraic set.
In this case the maximal rank of the rigidity matrix can be different on each irreducible component, and the realizations of maximal~rank~might no
longer be dense in $\real(\mathcal P)$. 
One option to deal with this is to accept that rigidity is no longer a generic property. 
A second option is to provide a notion of ``generic'' that works ``per irreducible component''. 
We go with the latter:

\begin{definition}
    \label{def:generic}
    A realization $(\bs p,\bs a)\in \real(\mathcal P)$ is \Def{generic} if it
    has maximal rank among all realizations that lie in the same irreducible component(s) of $\real(\mathcal P)$.
\end{definition}

When considering each component~of an irreducible decomposition of $\real(\mathcal{P})$ separately, this definition is consistent with the standard algebro-geometric notion of genericity that is, for example, given in \cite[Definition 5.6]{cox2005ideals}.
In \mbox{particular, the ge}\-neric realizations of $\mathcal P$ form an open and dense subset of $\real(\mathcal P)$. Lastly, \cref{def:generic}~also interacts well with the convex part of the realization space:

%

\begin{remark}
\label{res:dense_convex}
Recall that $\realCvx(\mathcal P)$ is open in its Zariski closure (in the Euclidean sense), and that the intersection of a dense set with a non-empty open set is dense in this open set.
From this we conclude: if $\mathcal P$ has convex realizations at all, then there are convex \emph{generic} realizations among them, and the latter are actually dense in $\realCvx(\mathcal P)$.    
\end{remark}

For defining ``generic rigidity'' we encounter a second problem: for non-simplicial $\mathcal P$ certain ``degenerate realizations'' can form irreducible components of realizations that are first-order flexible. 

\begin{example}
    Let $\mathcal P$ be the combinatorial type of the 3-dimensional cube.
    Consider the set $S\subset\real(\mathcal P)$ of realizations that are \textit{not} affinely spanning and note that here coplanarity constraints are vacuous.
    Hence, the (local) dimension of $S$ is 
    $$V\cdot\overbrace{(d-1)}^{\mathclap{\text{max.\ dimension of $\aff(\bs p)$}}}+ \underbrace{\dim \operatorname{Gr}(d,d-1)}_{\mathclap{\text{ways to choose $\aff(\bs p)$}}}=8\cdot 2+ 3=19,$$
    at almost all points (where $\operatorname{Gr}$ denotes the Grassmannian). 
    This is larger than~the dimension of the convex realization space, which is $E+6=18$ (see also \cite[Example 3.4]{rastanawi2021dimensions}). 
    Consequently, the Zariski closure of $\realCvx(\mathcal P)$ will not contain $S$.

    While we expect (and later prove in \cref{thm:generic}) that convex realizations of $\mathcal P$ are (almost always) first-order rigid, all realizations in $S$ are first-order flexible.
    To construct a flex $(\bsdot p,\bsdot a)$, fix $i\in V$, set $\dot p_j=0$ for all $j\in V\setminus\{i\}$, $\dot a_\sigma = 0$ for all $\sigma\in F$, and choose $\dot p_i\perp \aff(\bs p)$ non-zero and orthogonal to the affine span of $\bs p$.
\end{example}

We did not encounter this problem for simplicial polytopes because for them the realization space $\real(\mathcal P)\simeq (\RR^d)^V$ is irreducible and contains these degenerate~realizations as a proper algebraic (and hence, measure zero) subset.

The solution to this problem is to remove the irreducible components of ``degenerate realizations'' from consideration. 
To avoid a detailed discussion of what makes a realizations degenerate (\eg\ coplanar faces,
confluent vertices, etc.) and since we are mainly interested in convex realizations anyway, we go with the following: 

\begin{definition}
    \label{def:generically_rigid}
    A realization of $\mathcal P$ is \Def{Zariski convex} if it lies in the Zariski closure $\realCvxCl(\mathcal P)$ of the convex realization space.
    $\mathcal P$ is \Def{generically rigid} if each generic~Za\-riski convex realization is (first-order) rigid.
\end{definition}

With this definition of ``generically rigid'' the extent of \cref{res:generic} hinges on the generality of the notion ``Zariski convex''.
Zariski convex realizations include certain non-convex realizations, and we believe they are, in fact, quite general.
Especially in dimension three we consider it as the ``right'' analogue of the non-convexity permitted by Gluck's theorem.
An alternative description of the~class~of~Zariski~convex realizations is currently open.

\subsection{Polyhedra, planar graphs and their realizations}
\label{sec:polyhedra}

The combinatorics and geometry of polyhedra is much better behaved than that of general polytopes.
This is what eventually enables us to prove \cref{conj:generic} for $d=3$.
We recall here the most relevant facts (for details, see \cite[Chapter 4]{diestel2024graph} and \cite[Chapter 4]{ziegler2012lectures}).

The combinatorial type of a polyhedron is determined by its edge graph.
Moreover, \Def{Steinitz's Theorem} states that the graphs that appear as such edge graphs are precisely the 3-connected planar graphs, which are for this reason also known as \Def{polyhedral graphs}.
In the following we will use the graph $G$ instead of $\mathcal P$ to denote the combinatorial type of a polyhedron.

A polyhedral graph has an essentially unique planar embedding, in particular,
the notion of \Def{face} (as a connected component of the complement of the embedding; respectively its boundary cycle in the graph) is well-defined for the graph without specifying an embedding.
We will write polyhedral graphs as triples $G=(V,E,F)$ with vertex set~$V$, edge set $E$, and face set $F$.

In contrast to general polytopes, realization spaces of polyhedra are well-under\-stood.
Most relevant to us, $\realCvx(G)$ is an \emph{irreducible} semi-algebraic set, and~consequently, $\realCvxCl(G)$ is an \emph{irreducible} algebraic set (see \cite[Corollary 4.10]{rastanawi2021dimensions}).\nls
These simplified circumstances allow us to apply Gluck's trick almost verbatim:

\begin{lemma}
    \label{res:genericity_applied}
    Given a polyhedral graph $G$, the following are equivalent:
    \begin{myenumerate}
        \item $G$ is generically rigid (in the sense of \cref{def:generically_rigid}).
        \item there exists a convex realization of $G$ that is first-order rigid.
        \item there exists a Zariski convex realization of $G$ that is first-order rigid.
    \end{myenumerate}
\end{lemma}
\begin{proof}
    By Steinitz's Theorem, $G$ has a convex realizations
    and hence \itm1$\implies$\itm2 is a consequence of \cref{res:dense_convex}.
    The implication \itm2$\implies$\itm3 is clear.
    It remains to show $\neg$\itm1 $\Longrightarrow\neg$\itm3.

    Suppose that $G$ is \textit{not} generically rigid.
    Then there is a generic Zariski convex~realization $(\bs p,\bs a)\in\realCvxCl(G)$ that is first-order flexible.
    By \cref{res:rank_condition} the~corank of $(\bs p,\bs a)$ is larger than six.
    Since $(\bs p,\bs a)$ is generic, the corank of all realizations~in~the same irreducible component(s) of $\realCvxCl(G)$ is larger than six as well. 
    For polyhedral graphs 
    $\realCvxCl(G)$ is irreducible, and so this lower bound applies to all Zariski convex realizations of $G$.  
    Applying \cref{res:rank_condition} again yields that all Zariski convex realizations are first-order flexible. This shows $\neg$\itm3.
 \end{proof}

There are further specifically 3-dimensional tools that we will make use of in~the proof of \cref{res:generic}, most prominently, Tutte embeddings and the Maxwell-Cre\-mona correspondence. 
They are introduced later in \cref{sec:Tutte_MC}.

\subsection{Proof of \cref{thm:generic}}
\label{sec:proof_outline}
\label{sec:proof_full}

Throughout this section, we fix a polyhedral graph $G=(V,E,F)$.
To show that $G$ is generically rigid, by \cref{res:genericity_applied} it suffices~to~find a~single convex (actually, Zariski convex) realization of $G$ that is first-order rigid.
In contrast to the simplicial case, however, we do not have a counterpart of Dehn's theorem, that is, we cannot just point to an arbitrary convex realization. In fact, such a counterpart cannot exist since we already know that there are flexible polyhedra (\cf\ \cref{sec:examples}).

Instead we proceed by induction on the size of $G$, establishing the existence~of a first-order rigid realization by using the generic rigidity of all smaller polyhedral graphs.
The general strategy is rather simple, but there are several technical steps in between. 
We therefore first sketch the proof, then fill in the details step~by step, before we present the full proof of \cref{res:generic} in the end of this~section. Some particularly technical parts are postponed to \cref{sec:sequence} and \cref{sec:limit_is_stressed_framework}. 

The core of the proof is the induction step which is reminiscent of a ``reverse vertex-splitting argument'': for the following sketch, suppose that generic rigidity has been established for all polyhedral graphs with fewer vertices or edges than $G$.
Then proceed as follows:

\begin{enumerate}
    \setlength{\itemsep}{0.5ex}
    \item 
    \label{it:outline_e}
    Choose an edge $e\in E$ so that $\tilde G := G/e$ is still a polyhedral graph but with fewer vertices and edges than $G$. Here $G/e$ denotes \Def{edge contraction}, that is, we remove $e$, identify its former end vertices, and we delete any parallel edges that are created in the process.
    \item 
    \label{it:outline_P}
    Choose a generic convex polyhedral realization $\tilde P$ of $\tilde G$. By induction~hypo\-thesis we know that $\tilde P$ is first-order rigid. 
    \item 
    \label{it:outline_seq}
    Choose a sequence $P^1\!,P^2\!,P^3\!,...\subset\RR^3$ of convex polyhedral realizations of $G$ that~converges to \smash{$\tilde P$}.
    If, for sake of contradiction, we assume that there are no first-order rigid realizations of~$G$, then all $P^n$ are first-order flexible.
    \item 
    \label{it:outline_limit}
    Show that the limit of a sequence of first-order flexible polyhedra is~it\-self first-order flexible.
    Thus, if $P^n$ is first-order flexible for all $n\ge 1$,\nls then \smash{$\tilde P$}~is first-order flexible as well.
    \item 
    But $\tilde P$ was first-order rigid by the induction hypothesis. This yields a con\-tradiction. Hence, some realization $P^n$ must have been first-order rigid. 
    \item         
    Using \cref{res:genericity_applied}, $G$ is therefore generically rigid.
\end{enumerate}

Many steps in the above outline require further technical~elaborations or adjustments. We go through them step by step.

First, for \eqref{it:outline_e} we require that $G$ has an edge $e\in E$ that after contraction leaves us with a polyhedral graph $\smash{\tilde G=(\tilde V,\tilde E,\tilde F):=G/e}$.
This is guaranteed by a classical theorem due to Tutte:

\begin{theorem}[{\cite[Lemma 3.2.4.]{diestel2024graph}}]
    \label{res:Diestel}
    If $G$ is a 3-connected graph other than $K_4$, then there is an edge $e\in E(G)$ so that $G/e$ is still 3-connected.
\end{theorem}

We shall call such an edge \Def{contractible}.
In the following let $e\in E$ be such~a~contractible edge.
Since planarity is also preserved under edge contractions, we~conclude that $\tilde G:= G/e$ is still polyhedral. 

We encounter a first technical point: later (in step \eqref{it:outline_limit}) we need that $e$ is not~only contractible, but \Def{well-contractible}, which shall mean
\begin{myenumerate}
    \item $e$ is incident to a non-triangular face, or
    \item $e$ is \ul{not} incident to a vertex of degree three. 
\end{myenumerate}
Not every polyhedral graph has a well-contractible edge (see \cref{fig:not_well_contractible}).
For the~case that it has none we provide a lemma of alternatives. 
For this, say that a polyhedral graph $G$ has a \Def{stacked $K_4$} if it has a vertex $i\in V$ of degree three that is incident to three triangular faces (note that then $i$ and its neighbors form a $K_4$).

\begin{figure}
    \centering
    \includegraphics[width=0.35\linewidth]{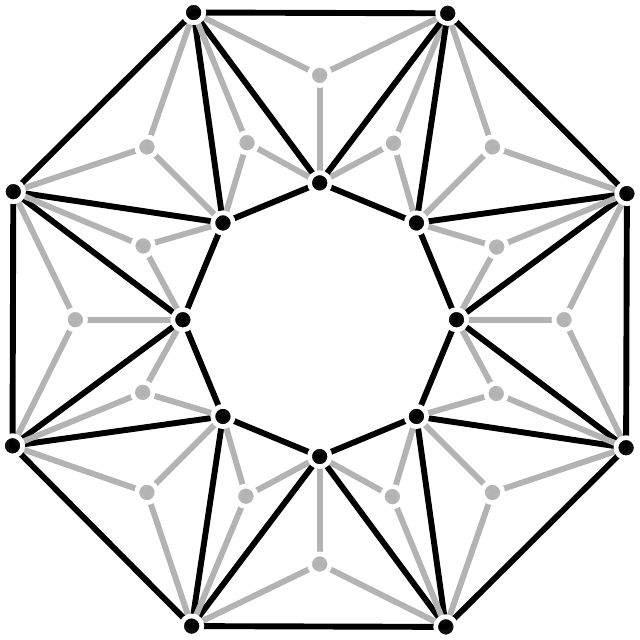}
    \caption{The figure shows the edge graph of an octagonal antiprism with a stacked $K_4$ on each triangular face. Note that through the stacking the edges of the antiprism are not contractible, whereas the edges of the stacked $K_4$'s shown in gray are contractible but not well-contractible.}
    \label{fig:not_well_contractible}
\end{figure}

\begin{lemma}
    \label{res:contraction_cases}
    For a polyhedral graph $G$, at least one of the following is true:
    \begin{myenumerate}
        \item $G=K_4$,
        \item $G$ has a well-contractible edge, or
        \item $G$ has a stacked $K_4$.
    \end{myenumerate}
\end{lemma}
\begin{proof}
    If $G\not= K_4$ then by \cref{res:Diestel} there is a contractible edge $\hatihatj\in E(G)$.
    Suppose that $\hatihatj$ 
    is not well-contractible, that is, $\hatihatj$ is incident to two triangles $\sigma_1,\sigma_2$, and incident to a vertex of degree three. W.l.o.g. we may assume that $\hati$ is the  degree three vertex.
    Let $i_k$ be the vertex of $\sigma_k\setminus\hatihatj$ for $k\in\{1,2\}$.
       The vertex $\hati$ is incident to three faces, two of which are $\sigma_1$ and $\sigma_2$. Let $\tau$ be the third face. 
    If $\tau$ is a triangle, then $\hati$, together with $\sigma_1$, $\sigma_2$ and $\tau$ form a stacked $K_4$, and \itm3  holds.
    If $\tau$ is not a triangle, then we claim that $\hati i_1$ is contractible, hence satisfies  \itm2.
    Suppose $\hati i_1$ is not contractible. Then $G/\hati i_1$ is 2-connected, and each 2-separator must use the contraction vertex and a third vertex, say $j$.
    Hence, $\hati$, $i_1$ and $j$ form a 3-separator of $G$. 
    The vertex $\hati$ must then have neighbors in different connected components of $G\setminus\{\hati,i_1,j\}$.
    But $\hati$ has only three neighbors, one of them, $i_1$, is in the separator,\nls and the other two, $i_2$ and $\hatj$, are adjacent.
    This yields a contradiction.
\end{proof}

In the following we shall assume that $e$ is well-contractible. The other case will be dealt with once we give the full proof of \cref{res:generic}.

Following the outline of the proof, we now choose a generic convex realization~$\tilde P$ of $\tilde G$ (which exists by \cref{res:dense_convex}).
In preparation for step \eqref{it:outline_limit} we once again need to make a technical assumption about our choice:
if $e$ is incident to precisely one non-triangular face $\sigma$ in $G$, then we require $\tilde P$ to 
lie ``above'' $\sigma$.
This means that the orthogonal projection of each vertex of $\tilde P\setminus\sigma$ onto $\sigma$ ends up in the relative interior of $\sigma$ (see \cref{fig:well_shaped}).
We call a convex realization $\tilde P$ \Def{well-shaped} if this is~the case.
Note that this can always be arranged using a projective transformation: choose a point $x$ outside $P$ but ``close to $\sigma$''; a suitable projective transformation now fixes $\sigma$ and moves $x$ to the infinite point orthogonal to $\sigma$ (this is a standard~step~in~the construction of Schlegel diagrams for polytopes, see for example \cite[Exercise 2.18]{ziegler2012lectures}).

\begin{figure}
    \centering
    \includegraphics[width=0.5\linewidth]{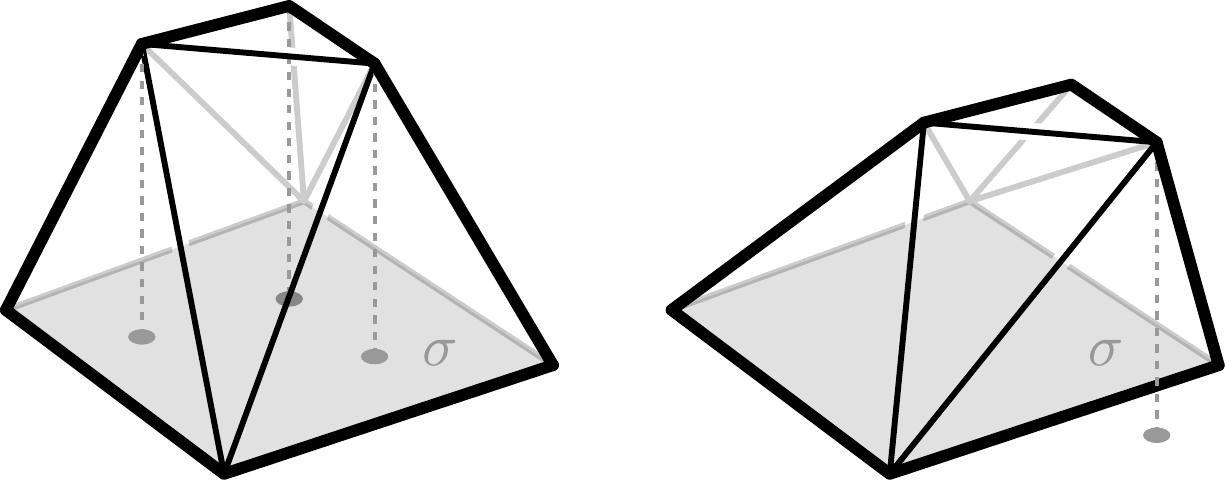}
    \caption{A well-shaped realization (left) and not well-shaped realization (right).}
    \label{fig:well_shaped}
\end{figure}

Moreover, since being well-shaped is clearly stable under small perturbations of the realization, we can guarantee that $\tilde P$ is both well-shaped and generic.
For later reference we state this as a proposition:

\begin{proposition}
    \label{res:well_shaped}
    There exists a convex realization $\tilde P$ of $\tilde G$ that is both generic and well-shaped.
\end{proposition}

In the following we assume that $\tilde P$ is a well-shaped generic realization of $\tilde G$.

In step \eqref{it:outline_seq} of the outline we choose a sequence $P^1\!,P^2\!, P^3\!,...$ of convex realizati\-ons of $G$ that converges to $\smash{\tilde P}$.
For the precise definition of convergence we~refer~to \cref{sec:sequence}. 
The existence of this sequence is not obvious, and is actually~very specific to dimension three.
The proof builds on Tutte embeddings and the Maxwell-Cremona correspondence (which we recall in \cref{sec:Tutte_MC}), as well as the limit behavior of these tools under topology changes (something that seems not available in the literature, see \cref{sec:appendix_MC}).
Due to this complexity, we moved the proof of existence to its own section (\cref{sec:sequence}).
Below we refer to it in the form of the following lemma:

\begin{lemma}
    \label{res:sequence}
    Given a contractible edge $e\in E$ and a convex polyhedral realization $\tilde P$ of \smash{$\tilde G$} $:= G/e$.
    There exists a sequence $P^1,P^2,$ $P^3,...\subset\RR^3$ of convex polyhedral realizations of $G$ with $P^n\to \tilde P$. 
\end{lemma}

The sequence given by \cref{res:sequence} will be called a \Def{contraction sequence} for the edge $e$.

Step \eqref{it:outline_limit} (the limit transfer of first-order flexibility) has a very technical proof as well.
We moved it to \cref{sec:limit_is_stressed_framework}.
It is here that we need our technical assumptions, that $e$ is well-contractible and that $\tilde P$ is well-shaped.
We refer to it~using~the~fol\-lo\-wing lemma:

\begin{lemma}
    \label{res:limit_is_stressed_framework}
    If $G$ is polyhedral, $e\in E$ is a well-contractible edge, $\tilde P$ is a well-shaped convex realization of $\tilde G:=G/e$, and $P^1,P^2,P^3,...$ $\to \tilde P$ is a contraction~sequence of first-order flexible convex polyhedra, then $\tilde P$ is first-order flexible as~well.
\end{lemma}

The remaining steps of the outline are straightforward. 
We now collect all these pieces to present the full proof of the main result:

\begin{theoremX}{\ref{res:generic_2}}
    If $G$ is the combinatorial type of a 3-polytope (\ie\ $G$ is a polyhedral graph), then it is generically (first-order) rigid.
\end{theoremX}
\begin{proof}
    The proof proceeds by induction on the number of vertices of $G$.%
    
    The minimal number of vertices of a polyhedral graph is attained for $G=K_4$. Its polyhedral realizations are tetrahedra, which are always first-order rigid (for~example due to Dehn's theorem, \cf\ \cref{res:comparison}).
    If $G$ has more than four vertices, then 
    by \cref{res:contraction_cases} there are two cases to consider.
    In both cases we construct a first-order rigid convex realization of $G$ and in this way show that $G$ is generically rigid by applying \cref{res:genericity_applied} \itm2$\implies$\itm1.

    \emph{Case 1: $G$ has a well-contractible edge $e$.}
    Then $\tilde G:= G/e$ is polyhedral with~fewer vertices than $G$. 
    By \cref{res:well_shaped} there is a convex realization $\tilde P$ of $\tilde G$ that is both generic and well-shaped.
    Since by induction hypothesis $\tilde G$ is generically rigid, $\tilde P$ is first-order rigid.
    By \cref{res:sequence} there is a contraction sequence $P^1\!,P^2\!,P^3\!,...$ of convex realizations of $G$ with $P^n\to\tilde P$.
    Suppose, for the sake of contradiction, that all $P^n$ are first-order flexible. 
    Since $\tilde P$ is well-shaped and 
    $e$ is well-contractible,\nls we can apply \cref{res:limit_is_stressed_framework} to conclude that $\tilde P$ is first-order flexible. 
    This is in contra\-diction to its initial choice. 
    We conclude that the sequence $P^n$ must contain first-order rigid polyhedra.

    \emph{Case 2: $G$ has a stacked $K_4$.}
    Let $i\in V(G)$ be the corresponding 3-vertex.
    Then $\tilde G:= G-i$ is still polyhedral with a distinguished triangular face $\Delta$. Since $\tilde G$ has fewer vertices than $G$, by induction hypothesis there is a first-order rigid convex realization $\tilde P$ of $\tilde G$.
    We obtain a convex realization $P$ of $G$ by stacking a sufficiently flat tetrahedron onto the face $\Delta$ of $\tilde P$. 
    We show that $P$ is first-order rigid:
    suppose that $(\dotbs p,\dotbs a)$ is a first-order motion of $P$.
    Since tetrahedra are first-order rigid, even as bar-joint frameworks, adding a suitable trivial motion to $\dotbs p$ ensures $\dot p_j=0$ for all $j\in V(\Delta)\cup\{i\}$.
    Let $\smash{(\bsdot p,\bsdot a)_{\tilde G}}$ be the restriction of this first-order motion to $\tilde G$.
    This yields a first-order motion of $\tilde P$ and must be trivial since $\tilde P$ is first-order rigid.
    But a trivial motion that vanishes on the triangular face $\Delta$ must be zero.
    Hence $(\dotbs p,\dotbs a)$ is zero.
    This shows that $P$ is first-order rigid.
\end{proof}

It remains to prove \cref{res:sequence} and \cref{res:limit_is_stressed_framework}.
We first recall some necessary tools and introduce suitable notation.

\subsection{Tutte embeddings and the Maxwell-Cremona correspondence}
\label{sec:TMC}
\label{sec:Tutte_MC}

Tutte's embedding theorem together with the Maxwell-Cremona correspondence constitute a powerful toolbox for controlling convex polyhedral realizations.
A great~source~dis\-cussing these tools is \cite[Section 3.2.1 and 3.2.2]{lovasz2019graphs}.
We recall the essentials.

A \Def{stress} on a framework $(G,\bs p)$ is a map $\bs\omega\: E\to\RR$. It is a \Def{self-stress} (sometimes \Def{equilibrium stress}) if at each vertex it satisfies the \emph{stress equilibrium condition}:
\begin{equation}
\label{eq:stress}
\sum_{\mathclap{j:ij\in E}} \omega_{ij}(p_j-p_i)=0,\quad\text{for all $i\in V$}.    
\end{equation}
A framework $(G,\bs p)$ together with a self-stress $\bs\omega$ we shall call a \Def{self-stressed framework} $(G,\bs p,\bs \omega)$.
A map $\hatbs\omega\:\hat E\to\RR$ defined on a subset of edges $\hat E\subset E$ we shall call a \Def{partial stress}. 

\begin{theorem}[Tutte embedding \cite{tutte1963draw}]
\label{res:Tutte}
\label{res:tutte}
Let $G=(V,E,F)$ be a polyhedral graph with a triangular face $\Delta\in F$.
For each positive partial stress \mbox{$\hatbs\omega\:E\setminus\Delta\to\RR_+$}~there exists a unique (up to linear transformations) 2-dimensional self-stressed framework $(G,\bs v,\bs \omega)$ whose self-stress $\bs\omega$ extends $\hatbs\omega$.
Moreover, $(G,\bs v)$ is a planar straight-line drawing (\ie\nls no two edges are crossing) with $\Delta$ as its outer face and all other faces convex.
\end{theorem}

A \Def{polyhedral lifting} of a 2-dimensional framework $(G,\bs p)$ is a triple $(G,\bs p,\bs h)$ with a \Def{height function} $\bs h\:V\to\RR$ so that for each face $\sigma\in F$ the points $(p_i,h_i)\in\RR^3,i\sim\sigma$ lie on a common plane.


\begin{theorem}[Maxwell-Cremona correspondence]
\label{res:Maxwell_Cremona}
\label{res:MC}
\label{res:maxwell_cremona}
\label{res:MCC}
Given a polyhedral graph $G$ and a 2-dimensional framework $(G,\bs v)$, there is a one-to-one correspondence between 
\begin{myenumerate}
    \item self-stressed frameworks $(G,\bs v,\bs \omega)$, and
    \item polyhedral liftings $(G,\bs v,\bs h)$ (up to certain projective transformations; or~uni\-quely if we prescribe $h_i=0$ for all vertices $i\sim\sigma$ on a fixed face $\sigma\in F$).
\end{myenumerate}
Moreover, if $\omega_e<0$ for all edges of some fixed face $\sigma\in F$, and $\omega_e>0$ for all other edges, then $(G,\bs v,\bs h)$ yields a convex realization of $G$.
\end{theorem}

For constructing the contraction sequence $P^n\to \tilde P$ in \cref{sec:sequence} we will use~that both Tutte embeddings and Maxwell-Cremona lifts behave well under edge contraction.
The precise meaning will be made clear at first use. The proofs are~straightforward but tedious and distract from the main proof. 
They are given in \cref{sec:topology_changes}.

\subsection{Contraction-only minors}
\label{sec:contraction_terminology}

This section introduces a language that is slightly more general than needed but will enable us to be more concise later on.

If $G$ is a graph, a \Def{contrac\-tion-only minor} $\tilde G\le G$ is a graph $\tilde G=(\tilde V,\tilde E)$ obtained from $G$ by contrac\-ting some of the edges, but \emph{not} by deleting edges.
Equivalently, $\tilde G=G/{\sim}_{\mathrm c}$ is a quotient \wrt\ an equivalence relation ${\sim}_{\mathrm c}$ on the vertices (the~``c'' is for ``contraction'') where each equivalence class induces a connected subgraph of $G$.
We shall use the vertices of $\tilde G$ interchangeably with these equivalence classes, which we denote by capital letters $I,J\in \tilde V$.
In the case of a single edge contraction $\hati\hatj\to\hatij$ (the most relevant case for us), the~contraction vertex~$\hatij$~cor\-responds to a class $I=\{\hati,\hatj\}$, and all other classes are singletons. Edges in $\tilde G$ are denoted~by~$IJ$ and we identify them with the set $\{ij\in E\mid i\in I\text{ and }j\in J\}$. 

Suppose now that both $G=(V,E,F)$ and $\tilde G=(\tilde V,\tilde E,\tilde F)$ are polyhedral.
We~say that a face $\sigma\in F$ \Def{persists} in $\tilde G$ if its vertices are part of at least three ${\sim}_{\mathrm c}$-equivalence classes.
If this is not the case then $\sigma$ \Def{collapses} either into an edge (if there are only two equivalence classes) or into a vertex (if there is only one equivalence class).
Since $\tilde G$ is a contraction-only minor (as opposed to a minor with edge deletions),\nls no two faces of $G$ can ``join'' into a single face in $\tilde G$. 
For this reason we can write $\tilde F\subseteq F$.

\subsection{Proof of \cref{res:sequence}: the existence of a contraction sequence}
\label{sec:sequence}

For this section we are given a polyhedral graph $G=(V,E,F)$ with an edge $\hati\hatj\in E(G)$ so that $\tilde G=(\tilde V,\tilde E,\tilde F):= G/\hatihatj$ is a polyhedral graph as well. 
The goal is to prove~the following:

\begin{lemmaX}{\ref{res:sequence}}
    Given a convex polyhedral realization $\tilde P$ of $\tilde G$, there exists a sequence $P^1,P^2,$ $P^3,...$ of convex polyhedral realizations of $G$ with $P^n\to \tilde P$.
\end{lemmaX}

The convergence $P^n\to \tilde P$ is meant in the following sense: if $P^n=(\bs p^n,\bs a^n)$ and $\tilde P=(\tilbs p,\tilbs a)$, then
\begin{align*}
    p_i^n\to \tilde p_I,&\quad\text{whenever $i\in I$,}
    \\[-0.8ex]
    a_\sigma^n \to \tilde a_\sigma,&\quad\text{whenever $\sigma\in\tilde F$}.
\end{align*}

Our strategy is to transform the problem into the language of planar self-stressed frameworks (via the Maxwell-Cremona correspondence, \cref{res:maxwell_cremona}), where we then construct the required sequence using Tutte embeddings (\cref{res:tutte}).

Recall that Tutte embeddings require a choice of facial triangle in $G$.
We additionally require that this triangle survives the edge contraction and persists in $\tilde G$.
We shall first assume that such a triangle exists, and deal with~the other case later.

\subsubsection*{Case 1: Assume that $G$ contains a triangular face $\Delta$ that persists in $\tilde G$.}
We can choose a projective transformation $T$ so that $\tilde Q:=T(\tilde P)$ is a Maxwell-Cremona lifting~of~a~planar self-stressed framework $(\tilde G,\tilbs v ,\tilbs\omega)$ whose outer face is $\Delta$, hence, with a positive self-stress on $E\setminus \Delta$.
In particular, $(\tilde G,\tilbs v ,\tilbs\omega)$ is also a Tutte embedding.

We next construct a sequence of self-stressed embeddings $(G,\bs v^n, \bs\omega^n)$ that ``con\-verges'' to $(\tilde G,\tilbs v,\tilbs\omega)$ in the sense
$$v_i^n\longrightarrow \tilde v_I,\quad\text{whenever $i\in I$}.$$
In particular, the edge $\hatihatj$ shrinks to length zero, while other edges stay of positive length.
This is achieved using Tutte embeddings, 
where the key observation is the following: if the self-stress on an edge goes to infinity, then the length of that edge in a Tutte embedding goes to zero.
This is intuitive: the edge models~a~spring~of~increasing spring constant, trying to contract ever stronger.
A~rigorous~formulation requires some care and can be found in \cref{sec:appendix_stresses} (specifically \cref{res:appendix_limit_stressed}).

Let $\hatbs\omega^n\: E\setminus \Delta\to\RR>0$ be a sequence of partial stresses with $\hat\omega_\hatihatj^n\to\infty$ and that converges to $\tilbs \omega$ on $E\setminus \Delta$ in the following sense: the limit $\hat\omega^n_{ij}\to\hat\omega^*_{ij}$ exists whenever $i\not\sim_{\mathrm c}j$ and
$$\sum_{\mathclap{\subalign{i &\in I \\ j &\in J}}} \hat\omega_{ij}^* =\tilde\omega_{IJ},\quad \text{for all $IJ\in \tilde E\setminus\Delta$}.$$
For example, for $ij\in IJ$ we could set $\hat\omega_{ij}^n=\hat\omega_\ij^*:=\tilde\omega_{IJ}/|IJ|$. 
Recall that $|IJ|$~denotes the number of elements of $\{ij\in E\mid i\in I\text{ and }j\in J\}$.

Let $(G,\bs v^n,\bs\omega^n)$ be a Tutte embedding extending the partial stress $\hatbs\omega^n$, and where we set $\smash{v_i^n:=\tilde v_{[\kern0.5pt i\kern0.7pt]}}$ if $i\in\Delta$ so as to make it unique (here $[\kern0.5pt i\kern0.7pt]$ denotes the equivalence class of $i$ in $\tilde G$).
We will now use that Tutte embeddings behave continuously in~the following sense:
the limit $\bs v^n\to\bs v^*$ exists, satisfies $v^*_i = \tilde v_I$ whenever $i\in I$, and is a Tutte embedding for the self-stresses $\tilbs\omega$. 
For details see \cref{sec:appendix_tutte} (specifically \cref{res:appendix_limit_tutte}).

Having built the contraction sequence for planar self-stressed frameworks, it~remains to lift it back to a contraction sequence of polyhedra.
Let then $Q^n\simeq(\barbs p^n,\barbs a^n)$ be the Maxwell-Cremona lift of $(G,\bs v^n,\bs \omega^n)$, where we assume $\smash{\bar p^n_i:=\bar p^n_{[\kern0.5pt i\kern0.7pt]}}$ for all $i\in\Delta$ so as to make the lifting unique.
Now we use that Maxwell-Cremona lifts behave continuously under limits.
The details can be found in \cref{sec:appendix_MC}.
We conclude that $Q^n\to \tilde Q$, and hence that $P^n:=T^{-1}Q^n$ is the desired sequence with $P^n\to P$.
This concludes Case 1.

\par\smallskip

For Case 1 we relied on the existence of a facial triangle that survives contraction.
The following lemma shows that if there is no such triangle then there must be a 3-vertex that survives contraction:

\begin{lemma}
    \label{res:triangle_3_vertex}
    Either $G$ contains a facial triangle that is not incident to $\hatihatj$, or $G$ contains a 3-vertex that is not incident to $\hatihatj$.
\end{lemma}
\begin{proof}
    At most two triangles are incident to $\hatihatj$, and at most two 3-vertices are~incident to $\hatihatj$.
    The result then follows from the claim that $G$ contains either at least three triangles or at least three 3-vertices.
    Suppose the contrary, that there are at most two triangles and 3-vertices, then 
    \begin{align*}
        2|E| &= \sum_{\mathclap{i\in V}} \mathrm{deg}(i)\ge 2\cdot 3+(|V|-2)\cdot4=4|V|-2 &&\hspace{-4em} \quad\implies\; |V|\le {\textstyle \frac{|E|+1}2},
        \\
        2|E| &= \sum_{\mathclap{\sigma\in F}} \operatorname{gon}(\sigma)\ge 2\cdot 3+(|F|-2)\cdot4=4|F|-2 &&\hspace{-4em} \quad\implies\; |F|\le {\textstyle \frac{|E|+1}2}.
    \end{align*}
    where $\operatorname{gon}(\sigma)$ denotes the \emph{gonality} of the polygon $\sigma$, \ie\ its number of edges.
    
    Using Euler's polyhedral formula $|V|-|E|+|F|=2$, we arrive at a contradiction:
    $$|E|+2=|V|+|F|\le \textstyle{\frac{|E|+1}2 + \frac{|E|+1}2}=|E|+1. \;\contradiction$$
\end{proof}

Therefore, if we are not in Case 1, then by \cref{res:triangle_3_vertex} there is a 3-vertex which will be our second and final case to consider.

\begin{figure}[h!]
    \centering
    \includegraphics[width=0.6\linewidth]{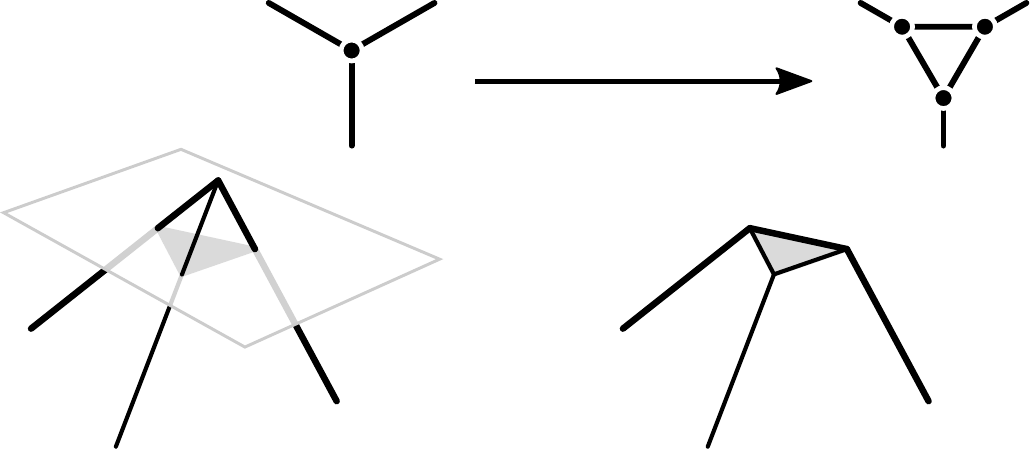}
    \caption{Transforming a 3-vertex into a triangular face.}
    \label{fig:cutting_off_vertex}
\end{figure}

\emph{Case 2: $G$ contains a 3-vertex that persists in $\tilde G$.}
Let then $i\in V\setminus\{\hati,\hatj\}$~be a 3-vertex.
The following operation on $i$ yields a polyhedral graph $H$ with a triangular face $\Delta$: subdivide all edges incident to $i$ with new vertices $j_1,j_2,j_3$, delete $i$, and add edges to form the triangle $\Delta:=j_1j_2j_3$. 
Geometrically this corresponds to ``cutting off'' the 3-vertex using a hyperplane that separates it from all other vertices (see \cref{fig:cutting_off_vertex}).
Since $\hatihatj$ is not incident to the 3-vertex $i$, this~operation can be equally performed in $\tilde G$ to construct a graph $\tilde H$, and we have $\tilde H = H/\hatihatj$. 

Given any polyhedral realization $\tilde Q$ of $\tilde H$, using Case 1 there is a sequence~$Q^1$, $Q^2, Q^3, ...$ of polyhedral realizations of $H$ with $Q^n\to\tilde Q$.
For $n$ sufficiently large, the inequalities defining the triangular face in $Q^n$ can be removed and the resulting polyhedron, call it $P^n$, is bounded and necessarily a realization of $G$. 
This yields the desired sequence $P^n\to \tilde P$.

\par\medskip
This finalizes the proof of \cref{res:sequence}. \hfill\qed

\begin{remark}
    \label{rem:general_contraction}
   It seems plausible that \cref{res:sequence} can be extended to general polyhedral minors $\tilde G\le G$ (\ie\ allowing edge deletions and more than one contraction).
   This result would then allow us to prove Case 2 using the dual polyhedron instead of cutting off a vertex (\cf\ \cref{fig:cutting_off_vertex}) using that edge contractions turn into edge deletions in the dual.
   However, this approach would also rely on extending the continuity results of Tutte embeddings and Maxwell-Cremona lifts, so it is beyond the scope of this paper.
\end{remark}

\subsection{Proof of \cref{res:limit_is_stressed_framework}: first-order flexibility in the limit}
\label{sec:limit_is_stressed_framework}

We recall the statement to be proven:

\begin{lemmaX}{\ref{res:limit_is_stressed_framework}}
    If $G=(V,E,F)$ is polyhedral, $e\in E$ is a well-contractible edge,\nls $\tilde P$~is a well-shaped convex realization of $\tilde G=(\tilde V,\tilde E,\tilde F):=G/e$, and $P^1,P^2,P^3,...$ $\to \tilde P$ is a~con\-traction sequence of first-order flexible convex polyhedra, then $\tilde P$ is first-order flexible as well.
\end{lemmaX}

By assumption, each $P^n$ has a first-order flex $(\dotbs p^n,\dotbs a^n)$.
Since they are non-zero, we can rescale each $(\dotbs p^n,\dotbs a^n)$ so that 
$$\sum_{i\in V} \|\dot p_i^n\|^2+\sum_{\sigma\in F} \|\dot a_\sigma^n\|^2=1.$$
This rescaled sequence lies in the compact unit sphere of $\RR^{dV+dF}$\!, and so it contains a convergent subsequence to which we can restrict.
We can now assume that~the limit $(\dotbs p^n,\dotbs a^n)\to (\dotbs p^*,\dotbs a^*)$ exists and is non-zero.
We set
$$
\dot{\tilde p}_I := \begin{cases}
    \dot p_i^* & \text{if $I\not=\hatij$}
    \\
    0 & \text{if $I=\hatij$}
\end{cases},
\;\;\text{for all $I\in \tilde V$}
,\qquad
\dot {\tilde a}_\sigma := \dot a_\sigma^*,\;\text{whenever $\sigma\in\tilde F$}.
$$
We verify that $(\dot{\tilbs p},\dot{\tilbs a})$ is a~non-zero first-order motion of the contracted polyhedron $\tilde P$.
This follows mostly from the fact that conditions \eqref{eq:flex_at_vertex} -- \eqref{eq:flex_at_face} hold for $(\dotbs p^n, \dotbs a^n)$ and are~preserved in the limit. 
To deal with topology changes in verifying \eqref{eq:flex_at_vertex} and \eqref{eq:flex_at_vertex_face} we use that $\dot p_\hati^n = \dot p_\hatj^n = 0$.
It now remains to show that $(\dot{\tilbs p},\dot{\tilbs a})$ is a \emph{non-trivial} first-order motion.

For this we need to make some retroactive assumptions about the contraction sequence and choice of flexes.
Let $\sigma$ and $\tau$ be the two faces incident to $\hatihatj$.
Further\-more, fix a line $\ell\subset\RR^3$.
We shall assume that the polyhedra $P^n$ are oriented so~that $p_\hati^n,p_\hatj^n\in\ell$ for all $n\in\NN$. 
Moreover, since $(\dotbs p^n,\dotbs a^n)$ is non-trivial, we can add to it~a suitable trivial first-order motion to achieve $\dot p_\hati^n=\dot p_\hatj^n=\dot a_\sigma^n=0$, while~$(\dotbs p^n,\dotbs a^n)$~remains non-zero.
Clearly, these modifications preserve generality and $\smash{(\dot{\tilbs p},\dot{\tilbs a})}$ can be defined as before.
Our strategy is now to assume that $\smash{(\dot{\tilbs p},\dot{\tilbs a})}$ is trivial, and deduce that it must also be zero, in contradiction to its construction.

If $\sigma$ is a triangle then the third vertex (next to $\hati$ and $\hatj$) will be called $i_\sigma$. Likewise $i_\tau$ for $\tau$.
For later use we now collect some useful relations for and between $(\bs p^n,\bs a^n)$ and $(\dotbs p^n,\dotbs a^n)$ which are preserved in the limit $n\to\infty$:

\par\smallskip
\begin{newmyenumerate}
\renewcommand{\labelenumi}{(R\arabic{enumi})}
    \setlength{\itemsep}{1ex}
    \item 
    $a_\sigma^n,a_\tau^n\perp\ell$ because $a_\sigma^n$ and $a_\tau^n$ are normal vectors to faces of $P^n$ that contain a segment of $\ell$ (namely, the edge $\hatihatj$).
    \item 
    If $\dot a_\tau^n\not=0$ then $\dot a_\tau^n\perp \ell$ is a direct consequence of the first-order flex condition \eqref{eq:flex_at_vertex_face}, that is, $0=\<\dot a_\tau^n,p_\hati^n-p_\hatj ^n\> +\<a_\tau^n,\dot p_\hati^n-\dot p_\hatj ^n\> = \<\dot a_\tau^n,p_\hati^n-p_\hatj ^n\>$, where we applied $\dot p_\hati ^n=\dot p_\hatj ^n=0$.
    \item 
    If $\dot a_\tau^n\not=0$ then $\dot a_\tau^n\perp a_\tau^n$ is a direct consequence of the first-order flex~condition \eqref{eq:flex_at_face}, that is, $\<\dot a_\tau^n, a_\tau^n\>=0$.
\end{newmyenumerate}

\par\smallskip
\noindent
If $\tau$ is a triangle we have the following additional relations:

\par\smallskip
\begin{newmyenumerate}
\renewcommand{\labelenumi}{(R\arabic{enumi})}
    \setcounter{enumi}{3}
    \setlength{\itemsep}{1ex}
    \item 
    $a_\tau^n \perp p_{i_\tau}^n-p_\hati^n$ because $p_{i_\tau}^n-p_\hati^n$ is the direction of an edge incident to $\tau$.\nls 
    The same clearly holds with $\hatj$ in place of $\hati$.
    \item 
    If $\dot p_{i_\tau}^n\not=0$ then $\dot p_{i_\tau}^n\perp p_{i_\tau}^n - p_\hati ^n$ is a direct consequence of the first-order flex condition \eqref{eq:flex_at_vertex}, that is,  $0=\<p_{i_\tau}^n - p_\hati ^n, \dot p_{i_\tau}^n - \dot p_\hati ^n\>=\<p_{i_\tau}^n - p_\hati ^n, \dot p_{i_\tau}^n\>$, where we applied $\dot p_\hati ^n=0$.
    The same clearly holds with $\hatj$ in place of $\hati$.
    \item
    If $\dot p_{i_\tau}^n\not=0$ then $\dot p_{i_\tau}^n\parallel a_\tau^n$ is a consequence of combining (R4) and (R5): both $\dot p_{i_\tau}^n$ and $a_\tau^n$ are perpendicular to the edges $i_\tau\hati$ and $i_\tau\hatj$.
\end{newmyenumerate}

\par\medskip
We now distinguish three cases:

\par\medskip
\emph{Case 1: both $\sigma$ and $\tau$ are non-triangular.}
Then both faces persist in the limit.\nls
We have $a_\sigma^*,a_\tau^*\perp\ell$ by (R1).
Since the faces persist, they cannot be parallel in the limit, hence $a_\sigma^*\nparallel a_\tau^*$. In other words, $a_\sigma^*$, $a_\tau^*$ and (the direction vector of) $\ell$ form a basis of $\RR^3$.
At the same time, if $\dot a_\tau^*\not=0$, then $\dot a_\tau^*\perp a_\tau^*,\ell$ by (R2) and (R3) respectively.
Hence, $\dot a_\tau^*$ cannot also be perpendicular to $a_\sigma^*$.

However, if $(\dotbs p^*,\dotbs a^*)$ is a trivial first-order motion, then condition \eqref{eq:trivial_motion_facet} applies to the persisting faces $\sigma,\tau\in F$ and simplifies to
$$0=\<a_{\sigma}^*,\dot a_{\tau}^*\> + \<\dot a_{\sigma}^*, a_{\tau}^*\> = \<a_{\sigma}^*,\dot a_{\tau}^*\>,\quad\text{(we used $\dot a_\sigma^*=0$)}$$
in contradiction to $\dot a_\tau^*\not\perp a_\sigma^*$. We conclude $\smash{\dot{\tilde a}_{\tau}}=\dot a_\tau^*=0$.

But a trivial first-order motion with $\smash{\dot {\tilde p}_{\hatij}=\dot {\tilde a}_\sigma=\dot {\tilde a}_\tau=0}$ must be zero: $\smash{\dot {\tilde p}_{\hatij}=0}$~prevents infinitesimal translation; $\smash{\dot {\tilde a}_\sigma=\dot {\tilde a}_\tau=0}$ pre\-vents infinitesimal rotation.

\par\medskip
\emph{Case 2: precisely one of $\sigma$ and $\tau$ is a triangle.} 
We may assume w.l.o.g.\ that $\tau$ is the triangle (with third vertex $i_\tau$).
In particular, $\sigma$ persists in the limit.

If $(\dotbs p^*,\dotbs a^*)$ is a trivial first-order motion, then condition \eqref{eq:trivial_motion_vertex_facet} applies to the~vertices $\hati ,i_\tau\in V$ together with the persisting face $\sigma\in F$ and simplifies to
$$0=\<a_{\sigma}^*,\dot p_{i_\tau}^* -\dot p_\hati^* \> + \<\dot a_{\sigma}^*,p_{i_\tau}^* - p_\hati^*\> = \<a_{\sigma}^*,\dot p_{i_\tau}^* \>. \quad\text{(we used $\dot a_\sigma^*=\dot p_\hati ^*= 0$)}$$
If $\dot p_{i_\tau}^*\not=0$, then this becomes $\dot p_{i_\tau}^* \perp a_\sigma^*$.
Together with $\dot p_{i_\tau}^*\parallel a_\tau^*$ from (R5) we obtain $a_\sigma^*\perp a_\tau^*$.
Note that even though $a_\tau^*$ is not the normal vector of a face of $\smash{\tilde P}$, it~is still the normal vector of a supporting hyperplane that contains the edge $i_\tau \hatij$.
However, by assumption $\tilde P$ is \emph{well-shaped}, that is, $\tilde p_{i_\tau}$ lies ``above'' $\sigma$.
Hence, no such supporting hyperplane can be orthogonal to $\sigma$.
We conclude $\smash{\dot{\tilde p}_{i_\tau}}\!=\dot p_{i_\tau}^* \!=0$.

But a trivial first-order motion with $\smash{\dot {\tilde p}_{\hatij}\,=\dot {\tilde p}_{i_\tau}\!=\dot {\tilde a}_\tau=0}$ must be zero: by $\smash{\dot {\tilde p}_{\hatij}=\dot {\tilde p}_{i_\tau}}\! = 0$ the infinitesimal motion must fix the line through $\tilde p_\hatij$ and $\tilde p_{i_\tau}$; $\smash{\dot {\tilde a}_\tau=0}$ prevents the remaining infinitesimal rotation around this line.

\par\medskip
\emph{Case 3: both $\sigma$ and $\tau$ are triangles.} 
If $(\dotbs p^*,\dotbs a^*)$ is a trivial first-order motion,\nls then condition \eqref{eq:trivial_motion_vertex} applies to the pair $i_\sigma,i_\tau\in V$ and simplifies to
$$0=\<p_{i_\tau}^*-p_{i_\sigma}^*,\dot p_{i_\tau}^*-\dot p_{i_\sigma}^*\> = \<p_{i_\tau}^*-p_{i_\sigma}^*,\dot p_{i_\tau}^*\>. \quad\text{(we used $\dot p_{i_\sigma}^*= 0$)}$$
If $\dot p_{i_\tau}^*\not=0$ then $\dot p_{i_\tau}^*\parallel a_\tau^*$ by (R6), and we obtain $\<p_{i_\tau}^*-p_{i_\sigma}^*,a_\tau^*\>=0$.
We also~have $\<p_{\hati}^*-p_{i_\tau}^*,a_\tau^*\>=0$ from (R4).
We conclude that $\tilde p_{i_\sigma}$, $\tilde p_{i_\tau}$ and $\tilde p_{\hatij}$ are coplanar.\nls
Note that even though $a_\tau^*$ does not define a face of $\smash{\tilde P}$, it still defines a supporting hyperplane that contains  $\tilde p_{i_\sigma}$, $\tilde p_{i_\tau}$ and $\tilde p_{\hatij}$.
Since $\hatij\kern0.8pt i_\sigma$ and $\hatij\kern0.8pt i_\tau$ are distinct edges of $\smash{\tilde P}$ that lie in a common supporting hyperplane, they must be edges of a common face $\kappa$ of~$\smash{\tilde P}$.\nls 
But then $\kappa$ already existed in $P$, and there it must have been adjacent to both $\sigma$ and $\tau$.
Since $\sigma$ and $\tau$ are themselves adjacent in $P$, this can only have been the case if these three faces share a vertex, necessarily of degree three, which must be either $\hati$ or $\hatj$.
This is in contradiction to our assumption that $\hatihatj$ is \textit{well-contractible}, that is, is either incident to a non-triangular face or non-incident to a 3-vertex.
We conclude that $\smash{\dot{\tilde p}_{i_\tau}}=\dot p_{i_\tau}^*=0$.

But a trivial first-order motion with $\smash{\dot {\tilde p}_{\hatij}\,=\dot {\tilde p}_{i_\tau}\!=\dot {\tilde p}_{i_\sigma}\!=0}$ is zero: the three vertices $\tilde p_{i_\sigma}$, $\tilde p_{i_\tau}$ and $\tilde p_{\hatij}$ do not lie on a line because they define two distinct edges, that is, they span a plane. A trivial first-order motion fixes this plane only if it is zero.

\par\medskip
This finalizes the proof of \cref{res:limit_is_stressed_framework}. \qed

\par\medskip

We comment here on a failed attempt at shortening the proof.
It is well known that every polyhedral graph contains a $K_4$-minor \cite[Section 11.3]{grunbaum1967convex}.
Thus, instead of contracting one edge at a time, it appears plausible to circumvent induction by contracting $G$ into $K_4$ in one step.
We do believe that constructing a corresponding contraction sequence $P^1,P^2,P^3,\dots\to\tilde P$ is still possible, though certainly more involved. 
However, the obstacle lies in the proof of this section: it~is~possible that a sequence of first-order flexible polyhedra converges to a first-order rigid polyhedron:
\begin{example}
\label{ex:flexible_to_rigid}
    Let $P$ be some first-order rigid polyhedron and $Q$ some other~generically oriented polyhedron.
    From \cref{sec:minkowski_sums} we know that $R(t):=P+ tQ$ is flexible for every $t>0$, in particular, first-order flexible.
    For $t\to 0$ the polytope $R(t)$ converges to $P$, which is first-order rigid.
\end{example}
As far as we can tell, proceeding inductively and contracting one edge at a time cannot be avoided within our approach to \cref{res:generic}.

\section{Further discussion and outlook}
\label{sec:connclusions}
\label{sec:outlook}

\subsection{Higher dimensions}
\label{sec:higher_dim}

Our proof of \cref{thm:generic} uses tools that are available only for 3-dimensional polytopes and have no direct analogues in dimension $d\ge 4$. 
Below we present a template for an argument that might eventually~be~used~to~obtain similar results in higher dimensions.

Let $\mathcal P$ be the combinatorial type of a convex $d$-polytope with $d\ge 4$. The goal is to show that it is generically rigid by induction on the dimension:
\begin{enumerate}
    \item Fix a generic realization $P$ of $\mathcal P$.
    \item Conclude that then the facets of $P$ are generic $(d-1)$-polytopes.
    \item Any first-order motion $\dotbs p$ of $P$ induces a first-order motion on the facets.
    \item Since the facets are generic, this motion is trivial by induction hypothesis.
    In other words, $\dotbs p$ restricts to a trivial motion on each facet.
    \item This implies that $\bsdot p$ is trivial by the \emph{Cauchy-Dehn theorem} (see \eg \cite[Theorem 8.1 and 8.6]{whiteley1984infinitesimally}).
\end{enumerate}
The gap in the argument lies in step (2): it is far from clear that a generic realization of $P$ guarantees that its facets are generic as well.

In fact, given a $d$-polytope $P$ (with algebraic coordinates) there exists a polytope $Q$ of dimension $d+2$ (a so-called \emph{stamp polytope} of $P$) which has a face $P'\subset Q$ that in each realization of $Q$ is \emph{projectively equivalent} to $P$ \cite{dobbins2013antiprismless}.
In other words,\nls in~order to make the above approach work, it seems necessary for us to show that already some projective transformation of $P$ is 
first-order rigid.

Since all our examples of flexible polytopes have parallel edges (see~\cref{q:parallel_edges}), but a generically chosen projective transformation has no parallel edges, we think this is at least plausible:

\begin{conjecture}
    \label{conj:projective_generic}
    \label{conj:projective}
    Given a convex polytope $P\subset\RR^d$ of dimension $d\ge 3$, a generic projective transformation of $P$ is (first-order) rigid.
\end{conjecture}

Conveniently, it suffices to prove \Cref{conj:projective_generic} in dimension three.
The statement in higher dimensions follow via the above proof template. For step (2) we use that the facets of a generic projective transformation of $P$ are themselves generically projectively transformed.

\subsection{The dodecahedron and second-order rigidity for polytopes}

\;\;Among the Platonic solids only one remains with a non-obvious rigidity status: the regular \emph{dodecahedron}\footnote{Recall that the regular tetrahedron, octahedron and icosahedron are simplicial, hence rigid by Dehn's theorem; and the cube is a zonotope, hence flexible (\cf\ \cref{sec:zonotopes}).} (\cref{fig:dodecahedron}).
In general, determining whether an arbitrary polytope is rigid is theoretically and computationally difficult.
Already the dodecahedron~(arguably, a relatively small polytope) needs $3V+3F=96$ variables to~be~represented. 
We found this to be out of reach for exact solvers.
The question for its rigidity also attracted decent attention on MathOverflow \cite{MODodecahedronFlex}, but remained unanswered.
\begin{figure}[h!]
    \centering
    \includegraphics[width=0.215\textwidth]{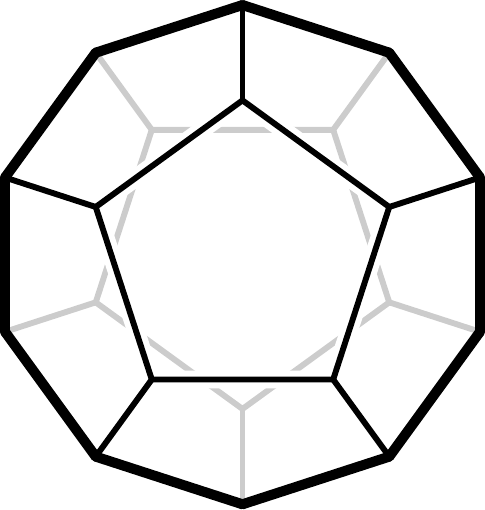}
    \caption{The regular dodecahedron.}
    \label{fig:dodecahedron}
\end{figure}

By now, the dodecahedron turned out to be an interesting toy example also~for other reasons.
First, it is \emph{not} first-order rigid, but has a 5-dimensional space of first-order flexes. Second, all these flexes seem to disappear after a generic affine transformation. This demonstrates that first-order rigidity of polytopes is not preserved under affine transformations -- a relevant point to consider in light of \cref{q:affine_trafo}.

The question for the rigidity of the regular dodecahedron was eventually resolved (through major contributions by Albert Zhang) by developing suitable \textit{second-order tools} for point-hyperplane frameworks. We therefore now know:
\begin{center}
    \emph{The regular dodecahedron is \ul{rigid}}.
\end{center}
The details of this will be reported in a separate paper focusing on the second-order theory of polytope rigidity.
%

\subsection{Edge length perturbations}


Consider a polytope $P\subset\RR^d$ with edge lengths $\ell_e,e\in E$.
Is there a polytope $P'$ combinatorially equivalent to $P$ with edge lengths $\ell_e':=\ell_e+\eps_e$ as long as $|\eps_e|$ is sufficiently small for all $e\in E$?
If so, we say that $P$ allows \emph{edge length perturbations}.

The fact that we need to respect the coplanarity constraints makes this appear unlikely. 
It turns out that in dimension $d=3$ this is a side effect of~infinitesimal rigidity: recall that $\realCvx(\mathcal P)$ for a 3-polytope is a smooth manifold of dimension $E+6$.
Suitably restricting the six trivial degrees of freedom (\eg\ by pinning or direction constraints) leaves us with a realization space of dimension $E$.
Infinitesimal rigidity ensures that the projection $\realCvx(\mathcal P)\ni(\bs p,\bs a)\mapsto(\|p_i-p_j\|)_{ij\in E}\in\RR^E$ is a local homeomorphism at $P$, meaning $P$ indeed allows edge length perturbations. 

\subsection{Other questions}

We are aware of a few other settings in which our notion~of rigidity poses interesting challenges and where analogous question can be asked:
\begin{itemize}
    \item non-convex polyhedra (\eg\ toroidal polyhedra, periodic polyhedra, etc.),
    \item non-orientable polyhedra,
    \item polyhedra in spherical or hyperbolic space.
\end{itemize}
In particular, at time of writing we are not aware of any flexible convex polyhedra in the non-Euclidean setting.

%% file: sec/appendix.tex
\section{Limit behavior of the Tutte-Maxwell-Cremona pipeline}
\label{sec:topology_changes}
\label{sec:MCC_Tutte_topology_changes}

In \cref{sec:sequence} we used that Tutte embeddings and Maxwell-Cremona lifts behave continuously in the limit of a contraction sequence.
We believe that this is natural, but we were unable to locate these statements in the literature.
We provide here a proof for general contraction sequences. See \cref{sec:Tutte_MC} to recall Tutte embeddings and Maxwell-Cremona lifts.

Throughout this appendix, let $G$ be a graph and $\tilde G\le G$ a contraction-only minor (see \cref{sec:contraction_terminology} to recall the notation and terminology of minors).

Given a sequence $(G,\bs v^n)$ of frameworks, we write $(G,\bs v^n)\to(G,\bs v^*)$ if \mbox{$\bs v^n\to\bs v^*$ as} $n\to\infty$. We call $(G,\bs v^*)$ the \Def{limit framework}.
We write $(G,\bs v^*)\simeq (\tilde G,\tilbs v)$~if~$v_i^*=\tilde v_I$ for all $i\in I$. 
In particular, this implies $v_i^*=v_j^*$ whenever $i\sim_{\mathrm c} j$. 
We call~$(\tilde G,\tilbs v)$ the \Def{contracted framework} and $(G,\bs v^n)$ a \Def{contraction sequence}.

Our goal is to show:
constructing the Tutte embedding or Maxwell-Cremona lift of a convergent sequence of input data and then taking the limit yields the~same~result as first taking the limit of the data and then constructing the Tutte embedding or Maxwell-Cremona lift thereof.
A schematic representation is given in \cref{fig:diagram}.

{
\begin{figure}[h!]
\centering
\begin{tikzcd}
\hat{\bs\omega}^n \arrow[r, "n\to\infty"] \arrow[d, "\mathrm T"'] 
& \hatbs\omega^* \arrow[r, "\sim"]       
& \hat{\tilbs\omega} \arrow[d, "\mathrm T"]     
\\
{(G,\bs v^n,\bs\omega^n)} \arrow[r, "n\to\infty"] \arrow[d, "\mathrm{MC}"']      
& {(G,\bs v^*,\bs\omega^*)} \arrow[r, "\sim"] 
& {(\tilde G,\tilbs v,\tilbs \omega)} \arrow[d, "\mathrm{MC}"] 
\\
{(G,\bs p^n,\bs h^n)} \arrow[r, "n\to\infty"]                      & {(G,\bs p^*,\bs h^*)} \arrow[r, "\sim"]      
& {(\tilde G,\tilbs p,\tilbs h)}                              
\end{tikzcd}    
\caption{Schematic representation of the limit behavior of the Tutte-Maxwell-Cremona pipeline as a commuting diagram. $\mathrm T$ stands~for~``constructing the Tutte embedding''; $\mathrm{MC}$ stands for ``constructing the~Max\-well-Cremona lift''.}
\label{fig:diagram}
\end{figure}
}

\subsection{Stresses}
\label{sec:appendix_stresses}
Suppose that $(G,\bs v^n,\bs\omega^n)$ is a sequence of self-stressed frameworks with $(G,\bs v^n)\to (G,\bs v^*)\simeq (\tilde G,\tilbs v)$.
Suppose further that $\omega_{ij}^n\to\omega_{ij}^*$ whenever $i\not\sim_{\text c} j$.
We can consider $\bs\omega^*$ as a partial stress on $G$ and write $(G,\bs v^n,\bs\omega^n)\to(G,\bs v^*,\bs\omega^*)$~to denote a convergence that includes the stress.

\begin{lemma}    
    \label{res:appendix_limit_stressed}
    In the setting above, $(\tilde G,\tilbs v,\tilbs\omega)$ is a self-stressed framework, where
    \begin{equation}
        \label{eq:def_omega_tilde}
        {
            \tilde\omega\:E(\tilde G)\to\RR,\quad
            \tilde \omega_{IJ} := \sum_{\mathclap{\subalign{i &\in I \\ j &\in J}}} \omega_{ij}^*
        }.
    \end{equation}
\end{lemma}
\begin{proof}
    We set $f_{ij}^n:=\omega_{ij}^n(v_j-v_i)$.
    Then $(*)\,\sum_j f_{ij}^n=0$ for all $i\in V(G)$.
    Moreover, $f_{ij}^n=-f_{ji}^n$, and for each $I\in V(\tilde G)$ holds (assuming some arbitrary order on $V(G)$)
    $$
    (**)\quad\sum_{\mathclap{i,j\in I}} f^n_{ij} 
        = \sum_{\mathclap{\substack{i,j\in I\\i<j}}} (f^n_{ij}+f^n_{ji}) 
        = \sum_{\mathclap{\substack{i,j\in I\\i<j}}} (f^n_{ij}-f^n_{ij})
        = 0
    $$
    Also, $f_{ij}^n\to f_{ij}^*:=\omega_{ij}^*(v_j^*-v_i^*)$ whenever $i\not\sim_{\text c} j$. For $I\not= J$ we set
    $$
    \tilde f_{IJ}
    :=\tilde\omega_{IJ}(\tilde v_J-\tilde v_I) 
    =\sum_{\mathclap{\subalign{i &\in I \\ j &\in J}}} \omega_{ij}^*(v_j^*-v_i^*)
    =\sum_{\mathclap{\subalign{i &\in I \\ j &\in J}}} f_{ij}^*.
    $$
    We can now verify the equilibrium condition in $(\tilde G,\tilbs v,\tilbs\omega)$: for each $I\in V(\tilde G)$ holds
    \begin{align*}
        \sum_{\mathclap{J:J\not= I}} \tilde f_{IJ}
        = 
        \sum_{\mathclap{J:J\not= I}}\,\, \sum_{\mathclap{\subalign{i&\in I\\j&\in J}}} f_{ij}^*
        =
        \sum_{i\in I} \sum_{\mathclap{j\not\in I}} f_{ij}^*
        \;\xleftarrow{n\to\infty}\;
        \sum_{i\in I} \sum_{\mathclap{j\not\in I}} f_{ij}^n
        =
        \smash{\sum_{i\in I}\underbrace{\sum_{j} f_{ij}^n}_{\mathclap{\text{$=\!0$ by $\!(*)$}\;\;}}
            - \!\!\underbrace{\,\,\sum_{\mathclap{i,j\in I}} f_{ij}^n}_{\mathclap{\;\;\text{$=\!0$ by $\!(**)$}}}}
        = 0.
    \end{align*}
    %
\end{proof}

Justified by \cref{res:appendix_limit_stressed} we will write $(G,\bs v^*,\bs\omega^*)\simeq (\tilde G,\tilbs v,\tilbs\omega).$

\subsection{Tutte embeddings}
\label{sec:appendix_tutte}

From this section on we assume that both $G$ and $\tilde G$ are polyhedral graphs.
For this section we furthermore assume that $G$ has a triangular face $\Delta\subset G$ that persists in $\tilde G$.

Let now $\hatbs\omega^n\:E\setminus\Delta\to\RR_+$ be a sequence of positive partial stresses on $G$, and let $(G,\bs v^n,\bs \omega^n)$ be a corresponding Tutte embedding (\cf\ \cref{res:tutte}).
Suppose $v_i^0,i\in\Delta$ are affinely independent, and fix $v_i^n:=v_i^0$ for all $i\in\Delta$, $n> 0$ to make a consistent choice for the embeddings throughout the sequence.

Suppose now that $\hat\omega^n_{ij}\to\hat\omega_{ij}^*\in\RR_+$ whenever $i\not\sim_{\mathrm c}j$ and $\hat\omega_{ij}^n\to\infty$ otherwise. We
define the partial stress
\begin{equation}
    \label{eq:contracted_partial_stress}
    \hat{\tilde{\bs\omega}}\:E(\tilde G)\setminus\Delta\to\RR_+,\quad
    \hat{\tilde{\omega}}_{IJ} := \sum_{\mathclap{\subalign{i &\in I \\ j &\in J}}} \hat\omega_{ij}^*>0.
\end{equation}

Let $(\tilde G,\tilbs v,\tilbs\omega)$ be the corresponding Tutte embedding with $\tilde v_{[i]}:= v_i^0$ for all $i\in\Delta$.
Our goal is to show that we obtain the same self-stressed framework if we take the limit of $(G,\bs v^n,\bs\omega^n)$ as $n\to\infty$, in particular, this limit exists in a relevant sense.
In other words, we show that the limit of Tutte embeddings is the Tutte embedding of the limit.

\begin{lemma}
    \label{res:appendix_limit_tutte}
    In the setting above there exists a framework $(G,\bs v^*,\bs\omega^*)$, with $\bs\omega^*$ a partial stress extending $\hatbs\omega^*$ to the edges of $\Delta$, so that $$(G,\bs v^n,\bs\omega^n)\to(G,\bs v^*,\bs\omega^*)\simeq(\tilde G,\tilbs v,\tilbs\omega).$$
\end{lemma}
\begin{proof}

Suppose for the proof that $V(\Delta)=\{i_1,i_2,i_3\}$. 

We need to show that the limit $\bs v^n\to\bs v^*$ exists.
Recall from \cref{res:tutte}~that $\bs v^n\subset\conv\{v_{i_1}^n,v_{i_2}^n,v_{i_3}^n\}$. 
Since the latter is a compact set independent of $n$, there exists a convergent subsequence $\bs v^{n_k}$.
We will later see that the limit of~${\bs v}^{n_k}$ is~independent of the chosen subsequence and hence that the limit $\bs v^n\to\bs v^*$ does indeed exist.
For convenience we shall for now write $\bs v^n$ instead of the convergent subsequence $\bs v^{n_k}$, which, as we just explained, will be retroactively justified.

To have $(G,\bs v^*)\simeq (G,\tilbs v)$ we certainly need $v_i^*=v_j^*$ whenever $i,j\in I$.
Suppose $I\in \smash{V(\tilde G)}$ is a counterexample.
Note that $I$ is finite, induces a connected subgraph of $G$, and, since $\Delta$ persists in $\tilde G$, contains at most one vertex of $\Delta$.
From this follow that there is an $i\in I\setminus\Delta$ for which $v_i^*$ is an extreme point (\ie\ vertex) of the convex polygon $\conv\{v_k^*\mid k\in I\}$, and that is furthermore adjacent to a $j\in I$ with $v_j^*\not= v_i^*$.
Since $v_i^*$ is an extreme point, there is a $c\in\RR^2$ so that $\<c,v_k^*-v_i^*\>\ge 0$ for all $k\in I$, an in particular $\<c,v_j^*-v_i^*\>> 0$.
Since $i\not\in \Delta$ the self-stress on all incident edges is given by $\hatbs\omega$.
Hence, the inner product of $c$ with the stress equilibrium conditions at $i$ in $\bs v^n$ becomes
\begin{align*}
0=
\Big\<c,\sum_{\mathclap{k:k\sim i}}\hat\omega_{ik}^n(v_k^n-v_i^n)\Big\>
&=
\sum_{\mathclap{k:k\sim i}}\hat\omega_{ik}^n\<c,v_k^n-v_i^n\>
\\[-0.5ex]&= 
\sum_{\mathclap{\substack{k:k\sim i\\k\sim_{\mathrm c}i}}}\overbrace{\hat\omega_{ij}^n}^{\mathclap{\to\infty}}\underbrace{\<c,v_j^n-v_i^n\>}_{\ge 0}
+ 
\sum_{\mathclap{\substack{k:k\sim i\\k\not\sim_{\mathrm c}i}}}\hat\omega_{ij}^n\<c,v_j^n-v_i^n\>
\end{align*}
The sum on the right converges (as established above) and hence the sum on the~left must converge as well. 
But since all terms in the left sum are positive and $\hat\omega_{ik}\to\infty$, this can only be the case if $\<c,v_k^*-v_i^*\>=0$ for all $k$ in the sum. But this contradicts $\<c,v_j^*-v_i^*\>>0$.
We conclude that $v_i^*=v_j^*$ for all $i,j\in I$ and that $(G,\bs v^*)\simeq (\tilde G,\tilbs v')$ for some $\tilbs v'$.
We later show $\tilbs v'=\tilbs v$.

We next show that the limit $\omega_{ij}^n\to\omega_{ij}^*$ exists also for $ij\in E(\Delta)$. 
For this, write $r_{ij}^n:=v_j^n-v_i^n$; then $r_{ij}^n\to r_{ij}^*:=v_j^*-v_i^*$ and the equilibrium condition at $i_1$ reads
$$
0
=
\sum_{\mathclap{j:j\sim i_1}} \omega_{i_1j}^n r_{i_1j}^n
=
\omega_{i_1i_2}^nr_{i_1 i_2}^n
+ \omega_{i_1i_3}^nr_{i_1 i_3}^n
+ \sum_{\mathclap{\substack{j:j\sim i_1\\ \phantom{j:}j\not\in\Delta}}} \hat\omega_{i_1j}^n r_{i_1j}^n.
$$
By rearranging we find
$$
\omega_{i_1i_2}^nr_{i_1 i_2}^n
+ \omega_{i_1i_3}^nr_{i_1 i_3}^n
= -\sum_{\mathclap{\substack{j:j\sim i_1\\ \phantom{j:}j\not\in\Delta}}} \hat\omega_{i_1j}^n r_{i_1j}^n
\xrightarrow{n\to\infty}
-\sum_{\mathclap{\substack{j:j\sim i_1\\ \phantom{j:}j\not\in\Delta}}} \hat\omega_{i_1j}^* r_{i_1j}^*.
$$
The right side is a finite number. 
Since $r_{i_1i_2}^n=r_{i_1i_2}^0$ and $r_{i_1i_3}^n=r_{i_1i_3}^0$ are \mbox{linearly~inde}\-pendent and constant, $\omega_{i_1i_2}^n$ and $\omega_{i_1i_3}^n$ are uniquely determined and convergent.
We therefore established that $\bs \omega^*$ exists.

So far we have seen $(G,\bs v^n)\to(G,\bs v^*)\simeq (\tilde G,\tilbs v')$ and $\omega_{ij}^n\to\omega_{ij}^*$ whenever $i\not\sim_{\mathrm c} j$.
In other words, we have shown that the preconditions of \cref{res:appendix_limit_stressed} are met.
We conclude that $(G,\bs v^n,\bs\omega^n)\to(G,\bs v^*,\bs\omega^*)\simeq (\tilde G,\tilbs v',\tilbs \omega')$, where $\tilbs \omega'$ is a self-stress~extending $\smash{\hat{\tilbs\omega}}$ (as defined in  \eqref{eq:contracted_partial_stress}).
Hence, it is a Tutte embedding.
By construction it also has $\tilde v_{[i]}' = v_i^0$ for all $i\in V(\Delta)$. 
Since the Tutte embeddings with prescribed vertices for the outer face is unique, we obtain $\smash{(\tilde G,\tilbs v',\tilbs \omega')=(\tilde G,\tilbs v,\tilbs \omega)}$.
It is at this point that we see that the limit $\bs v^*$ must indeed have been independent of the chosen convergent subsequence $\bs v^{n_k}$.
\end{proof}

\subsection{Reciprocal frameworks}
\label{sec:appendix_reciprocal}

To establish the limit behavior for the Maxwell-Cremona correspondence (\cref{res:maxwell_cremona}) we trace the path of its well-known proof using reciprocal frameworks.
In this version, given a 2-dimensional framework~$(G,\bs v)$ of a polyhedral graph $G$, one constructs a one-to one correspondence between
\begin{myenumerate}
    \item self-stressed frameworks $(G,\bs v,\bs \omega)$,
    \item reciprocal frameworks $(G^\star,\bs w)$ (up to translation), and
    \item polyhedral liftings $(G,\bs v,\bs h)$ (up to certain projective transformations).
\end{myenumerate}
In this section we carry the limit from \itm1 to \itm2.

Recall that to each polyhedral graph $G=(V,E,F)$ there is a \Def{dual graph} $G^\star=(F,E^\star,V)$, also polyhedral, whose vertices are the faces of $G$, two of which are~adjacent in $G^\star$ if they share an edge in $G$.
Moreover, each edge $ij\in E$ corresponds~to precisely one edge $\sigma\tau\in E^\star$, a correspondence which we express by $ij\perp \sigma\tau$.

Given a 2-dimensional framework $(G,\bs v)$, a \Def{reciprocal framework} is a framework of the form $(G^\star\!,\bs w)$ so that
$$\<v_j-v_i, w_\tau-w_\sigma\>=0,\quad\text{whenever $ij\perp\sigma\tau$}.$$
To obtain the one-to-one correspondence between self-stresses on $(G,\bs v)$ and reciprocal frameworks of $(G,\bs v)$ (up to translation) one shows that for a given self-stressed framework $(G,\bs v,\bs \omega)$ the system of equations
\begin{equation}
    \label{eq:to_reciprocal}
    w_\tau- w_\sigma=\omega_{ij} R_{\pi/2} (v_j-v_i),\quad\text{whenever $ij\perp\sigma\tau$}
\end{equation}
(where $R_{\pi/2}$ denotes a $90^\circ$-rotation) has a unique solution in $\bs w$ (up to translation), and that it yields a reciprocal framework $(G,\bs w)$.%
\footnote{One can moreover show that~$1/\bs\omega$ defines a self-stresses for the reciprocal framework, where by $1/\bs\omega$ we mean the map $\sigma\tau\mapsto 1/\omega_{ij}$ with $ij\perp\sigma\tau$.%
}

As before, let now $G$ and $\tilde G:=G/\!\sim_{\mathrm c}$ be polyhedral graphs, and  let $G^\star$ and~$\tilde G^\star$ be their corresponding dual graphs.
Let $\smash{(G,\bs v^n,\bs \omega^n)\to (G,\bs v^*,\bs \omega^*)\simeq (\tilde G,\tilbs v,\tilbs\omega)}$ be a contraction sequence of 2-dimensional self-stressed frameworks.
Let $(G^\star,\bs w^n)$ and $(\tilde G^\star,\tilbs w)$ be the corresponding reciprocal frameworks, where we fix $w_{\bar\sigma}^n=\tilde w_{\bar\sigma}$ for an arbitrary face $\bar\sigma\in F(\tilde G)$ so as to make a consistent choice throughout the sequence.
The goal is now to show that $(G^\star,\bs w^n) \to (\tilde G^\star,\tilbs w)$ in the sense $w_\sigma^n\to \tilde w_\sigma$ whenever $\sigma\in F(\tilde G)$.

\begin{lemma}
    In the setting above, there is a framework $(G^\star,\bs w^*)$ so that
    $$(G^\star,\bs w^n) \to (G^\star,\bs w^*)\simeq (\tilde G^\star,\tilbs w).$$
\end{lemma}
\begin{proof}
    Fix an edge $\sigma\tau\in E(\tilde G^\star)$ and let $IJ\in E(\tilde G)$ be the corresponding dual edge.
    
    We consider the following subgraph $H\subseteq G^\star$: a face $\rho$ of $G$ is a vertex of $H$ if it has an edge that lies in $IJ$; two faces $\rho,\rho'$ are adjacent in $H$ if they intersect in an edge in $IJ$.
    We make two observations:
    \begin{itemize}[leftmargin=2em]
        \item 
        $\sigma,\tau\in V(H)$ and are of degree one in $H$: if, say, $\sigma$ would have two neighbors $\rho,\rho'$ in $H$ then it would have two edges $e,e'\in IJ$ in $G$, and since it persists in $\tilde G$, it would have $IJ$ as an edge twice in $\tilde G$, which is impossible.
        \item 
        each $\rho\in V(H)$ that persists in $\tilde G$ is incident to $IJ$.
        But since $IJ$ is incident to exactly two faces in $\tilde G$, which are $\sigma$ and $\tau$, each other faces $\rho\in V(H)\setminus\{\sigma,\tau\}$ must have been collapsed in $\tilde G$. This means that all vertices of $\rho$ are in $I$ and $J$, and hence, that $\rho$ has an \emph{even} number of edges in $IJ$, that is, an even degree in $H$.
    \end{itemize}
    We can conclude that all vertices of $H$ are of even degree, except for $\sigma$ and $\tau$, which are of degree one.
    It follows from the handshaking lemma that then $\sigma$ and $\tau$ lie in the same connected component of $H$, that is, there is a path $\sigma_0,\sigma_1,...,\sigma_R$ from $\sigma=\sigma_0$ to $\tau=\sigma_R$ in $H$ (hence $G^\star$) where $\sigma_{r-1}\sigma_r$ form an edge.
 
    Let $e_r:=i_r j_r\perp \sigma_{r-1}\sigma_r$ be the dual edges with $i_r\in I$ and $j_r\in J$.
    For those \eqref{eq:to_reciprocal} becomes
    \begin{align*}
        w_{\sigma_{r}}^n-w_{\sigma_{r-1}}^n 
        =&\;\, 
        \omega_{i_r j_r}^nR_{\pi/2}(v_{j_r}^n-v_{i_r}^n) 
        \\
        \xrightarrow{n\to\infty}& \;\,
        \omega_{i_r j_r}^*R_{\pi/2}(v_{j_r}^*-v_{i_r}^*) 
        = \omega_{i_r j_r}^*R_{\pi/2}(\tilde v_J-\tilde v_I).
    \end{align*}
    Summing both sides over $r\in\{0,...,R\}$ (and then telescoping on the left) yields
    \begin{equation}
        \label{eq:limit_reciprocal}
        w_{\tau}^n-w_{\sigma}^n
        \xrightarrow{n\to\infty}  
        \sum_r \omega_{i_r j_r}^* R_{\pi/2}(\tilde v_J-\tilde v_I) 
        = \tilde\omega_{IJ} R_{\pi/2}(\tilde v_J-\tilde v_I).
    \end{equation}
    In particular, the limit of $w_{\tau}^n-w_{\sigma}^n$ exists.
    
    Since $\sigma\tau\in E(G^\star)$ was arbitrary, $G^\star$ is connected, and we fixed $w_{\bar\sigma}^n:=\tilde w_{\bar\sigma}$, this shows that the limit $w_\sigma^n\to w_\sigma^*$ exists for all $\sigma\in V(G^\star)$. 
    Moreover, by \eqref{eq:limit_reciprocal} $\bs w^*$ is a solution to \eqref{eq:to_reciprocal} applied to $(\tilde G,\tilbs v,\tilbs\omega)$.
    Since the solution to \eqref{eq:to_reciprocal} is unique up to translation, and since we also have $w_{\bar\sigma}^*=\tilde w_{\bar\sigma}$, we obtain $(G^\star,\bs w^*)\simeq(\tilde G^\star,\tilbs w)$. 
\end{proof}

\subsection{Maxwell-Cremona liftings}
\label{sec:appendix_MC}

In this last step we show that the limit behavior transfers from reciprocal frameworks to polyhedral lifts.
We recall that a lifting $\bs h$ is obtained from the reciprocal framework $(G^\star,\bs w)$ as a solution of the following system of equations:
\begin{equation}
\label{eq:to_lift}
h_j-h_i = \<w_{\sigma(ij)},v_j-v_i\>\quad\text{for all $ij\in E$},    
\end{equation}
where $\sigma(ij)$ is any face of $G$ incident to $ij$.
One can show that the solution to this system exists, is independent of the choice of $\sigma(ij)$, and is unique up to a choice of $h_{\bar\imath}$ for an arbitrary $\bar\imath\in V(G)$.

Suppose now that we are given 
$$(G,\bs v^n,\bs\omega^n)\to(G,\bs v^*,\bs\omega^*)\simeq(\tilde G,\tilbs v,\tilbs\omega)$$
and corresponding reciprocals
$$(G^\star,\bs w^n)\to(G^\star,\bs w^*)\simeq (\tilde G^\star,\tilbs w).$$
with $w^n_{\bar\sigma}=\tilde w_{\bar \sigma}$ for all $n$ (where $\bar\sigma\in F(G)$ is an arbitrary face). 

Let $(G,\bs v^n,\bs h^n)$ and $(\tilde G,\tilbs v,\tilbs h)$ be corresponding Maxwell-Cremona lifts (\cf\ \cref{res:maxwell_cremona}) with $\smash{h_{\bar \imath}^n = \tilde h_{[\bar\imath]}}$ for some arbitrary $\bar\imath\in V(G)$.
Our goal is to show that $(G,\bs v^n,\bs h^n)$ converges to $\smash{(\tilde G,\tilbs v,\tilbs h)}$ in the sense $v^n_i\to\tilde v_I$ and $\smash{h^n_i\to\tilde h_I}$ whenever~$i\in I$.

\begin{lemma}
    In the setting above there is a lifted framework $(G,\bs v^*,\bs h^*)$ so that
    $$(G,\bs v^n,\bs h^n) \to (G,\bs v^*,\bs h^*)\simeq (\tilde G,\tilbs v,\tilbs h).$$
\end{lemma}
\begin{proof}
    The convergence of $\bs v^n$ is assumed, and so we focus on $\bs h^n$. We have
    $$h_j^n-h_i^n = \<w_\sigma^n, v_j^n-v_i^n\> \xrightarrow{n\to\infty} \<w_\sigma^*, v_j^*-v_i^*\> = \<\tilde w_\sigma, \tilde v_J-\tilde v_I\>.$$
    In particular, $h_j^n-h_i^n$ converges for all $ij\in E$.
    Since $G$ is connected and we fixed $h_{\bar\imath}^n=\smash{\tilde h_{[\bar \imath]}}$, we obtain that the limit $h_i^n\to h_i^*$ exists for all $i\in V(G)$.
    Moreover, $\bs h^*$ satisfies \eqref{eq:to_lift} for $\smash{(\tilde G,\tilbs v)}$.
    Since $\bs h^*$ agrees with $\smash{\tilde h}$ on $\bar\imath$, and since the solution~to~\eqref{eq:to_lift} is unique up to the choice of a single value, we conclude $(G,\bs v^*,\bs h^*)\simeq\smash{(\tilde G,\tilbs v,\tilbs h)}$.
\end{proof}

%% file: reference.bib
@article {EJNSTW,
    AUTHOR = {Eftekhari, Y. and Jackson, B. and Nixon, A. and
              Schulze, B. and Tanigawa, S. and Whiteley, W.},
     TITLE = {Point-hyperplane frameworks, slider joints, and rigidity
              preserving transformations},
   JOURNAL = {J. Combin. Theory Ser. B},
  FJOURNAL = {Journal of Combinatorial Theory. Series B},
    VOLUME = {135},
      YEAR = {2019},
     PAGES = {44--74},

}

@article{connelly2005generic,
  title={Generic global rigidity},
  author={Connelly, Robert},
  journal={Discrete \& Computational Geometry},
  volume={33},
  pages={549--563},
  year={2005},
  publisher={Springer}
}

@book{ziegler2012lectures,
  title={Lectures on polytopes},
  author={Ziegler, G{\"u}nter M},
  volume={152},
  year={2012},
  publisher={Springer Science \& Business Media}
}

@misc{pak2010lectures,
  author = {Pak, Igor},
  title  = {Lectures on Discrete and Polyhedral Geometry},
  year   = {2010},
  note   = {Manuscript. \url{http://www.math.ucla.edu/~pak/book.htm}},
}

@article{rastanawi2021dimensions,
  title={On the dimensions of the realization spaces of polytopes},
  author={Rastanawi, Laith and Sinn, Rainer and Ziegler, G\"unter M},
  journal={Mathematika},
  volume={67},
  number={2},
  pages={342--365},
  year={2021},
  publisher={Wiley Online Library}
}

@incollection{gluck1975almost,
  title={Almost all simply connected closed surfaces are rigid},
  author={Gluck, Herman},
  booktitle={Geometric topology},
  pages={225--239},
  year={1975},
  publisher={Springer}
}

@article{connelly1977counterexample,
  title={A counterexample to the rigidity conjecture for polyhedra},
  author={Connelly, Robert},
  journal={Publications Math{\'e}matiques de l'IH{\'E}S},
  volume={47},
  pages={333--338},
  year={1977}
}

@book{alexandrov2005convex,
  title={Convex polyhedra},
  author={Alexandrov, Alexandr D},
  volume={109},
  year={2005},
  publisher={Springer}
}

@article{connelly2018affine,
  title={Affine rigidity and conics at infinity},
  author={Connelly, Robert and Gortler, Steven J and Theran, Louis},
  journal={International Mathematics Research Notices},
  volume={2018},
  number={13},
  pages={4084--4102},
  year={2018},
  publisher={Oxford University Press}
}

@book{diestel2024graph,
  title={Graph theory},
  author={Diestel, Reinhard},
  year={2024},
  publisher={Springer (print edition); Reinhard Diestel (eBooks)}
}

@MISC{MODodecahedronFlex,
    TITLE = {Is the dodecahedron flexible (as a polytope with fixed edge-lengths)?},
    AUTHOR = {M. Winter (https://mathoverflow.net/users/108884/m-winter)},
    HOWPUBLISHED = {MathOverflow},
    NOTE = {URL:https://mathoverflow.net/q/434771 (version: 2022-11-17)},
    EPRINT = {https://mathoverflow.net/q/434771},
    URL = {https://mathoverflow.net/q/434771}
}

@article{dobbins2013antiprismless,
author = {Dobbins, Michael G.},
year = {2017},
title = {Antiprismless, or: Reducing Combinatorial Equivalence to Projective Equivalence in Realizability Problems for Polytopes},
volume = {57},
journal = {Discrete \& Computational Geometry},
doi = {10.1007/s00454-017-9874-y}
}

@article{whiteley1984infinitesimally,
  title={Infinitesimally rigid polyhedra. {I.} Statics of frameworks},
  author={Whiteley, Walter},
  journal={Transactions of the American Mathematical Society},
  volume={285},
  number={2},
  pages={431--465},
  year={1984}
}

@article{connelly1980rigidity,
  title={The rigidity of certain cabled frameworks and the second-order rigidity of arbitrarily triangulated convex surfaces},
  author={Connelly, Robert},
  journal={Advances in Mathematics},
  volume={37},
  number={3},
  pages={272--299},
  year={1980},
  publisher={Academic Press}
}

@article{connelly2017prestress,
  title={Prestress stability of triangulated convex polytopes and universal second-order rigidity},
  author={Connelly, Robert and Gortler, Steven J},
  journal={SIAM Journal on Discrete Mathematics},
  volume={31},
  number={4},
  pages={2735--2753},
  year={2017},
  publisher={SIAM}
}

@article{husty2007nine,
  title={On a nine-bar linkage, its possible configurations and conditions for paradoxial mobility},
  author={Husty, Manfred L and Walter, Dominik},
  journal={Proc. IFFToMM, Besancon, France},
  year={2007}
}

@article{wunderlich1976deformable,
title = {On deformable nine-bar linkages with six triple joints},
journal = {Indagationes Mathematicae (Proceedings)},
volume = {79},
number = {3},
pages = {257-262},
year = {1976},
issn = {1385-7258},
doi = {10.1016/1385-7258(76)90052-4},
url = {https://www.sciencedirect.com/science/article/pii/1385725876900524},
author = {Walter Wunderlich}
}

@book{richter2006realization,
  title={Realization spaces of polytopes},
  author={Richter-Gebert, J{\"u}rgen},
  year={2006},
  publisher={Springer}
}

@book{grunbaum1967convex,
  title={Convex polytopes},
  author={Gr{\"u}nbaum, Branko},
  volume={16},
  year={1967},
  publisher={Springer}
}

@book{lovasz2019graphs,
  title={Graphs and geometry},
  author={Lov{\'a}sz, L{\'a}szl{\'o}},
  volume={65},
  year={2019},
  publisher={American Mathematical Soc.}
}

@article{tutte1963draw,
  title={How to draw a graph},
  author={Tutte, W. T.},
  journal={Proceedings of the London Mathematical Society},
  volume={3},
  number={1},
  pages={743--767},
  year={1963},
  publisher={Oxford University Press}
}

@article{bricard1897memoire,
  title={M{\'e}moire sur la th{\'e}orie de l'octa{\`e}dre articul{\'e}},
  author={Bricard, Raoul},
  journal={Journal de Math{\'e}matiques pures et appliqu{\'e}es},
  volume={3},
  pages={113--148},
  year={1897}
}

@article{zhai2018deployable,
author={Zirui Zhai and Yong Wang and Hanqing Jiang},
title={Origami-inspired, on-demand deployable and collapsible mechanical metamaterials with tunable stiffness},
journal = {Proceedings of the National Academy of Sciences},
volume={115},
number={9},
doi={10.1073/pnas.1720171115},
year={2018}
}

@article{MENG2023polyhedralmechanisms,
title = {Deployable polyhedral mechanisms with radially reciprocating motion based on novel basic units and an additive-then-subtractive design strategy},
journal = {Mechanism and Machine Theory},
volume = {181},
pages = {105174},
year = {2023},
issn = {0094-114X},
doi = {10.1016/j.mechmachtheory.2022.105174},
author = {Qizhi Meng and Fugui Xie and Ruijie Tang and Xin-Jun Liu},
}

@inproceedings{Sharifmoghaddam2024_thedra,
author = {Sharifmoghaddam, Kiumars and Mundilova, Klara and Nawratil, Georg and Tachi, Tomohiro},
year = {2024},
month = {},
pages = {},
title = {Woven Rigidly Foldable T-hedral Tubes Along Translational Surfaces},
booktitle={Proceedings of the 8th International Meeting on Origami in Science, Mathematics and Education-8OSME}
}

@article{Perlmutter2015viralselfassembly,
    author = {Jason D. Perlmutter and Michael F. Hagan},
    title = {Mechanisms of Virus Assembly},
    journal = {Annual Review of Physical Chemistry},
    year = {2015},
    volume = {66},
    doi = {10.1146/annurev-physchem-040214-121637}
}

@article{Parvez2020geometricvirus,
    author = {Mohammad K. Parvez },
    title={ Geometric architecture of viruses },
doi = {10.5501/wjv.v9.i2.5},
year = {2020},
volume = {9},
number = {2},
journal = {World Journal of Virology}
}

@article{izmestiev2024tsurfaces,
author = {Izmestiev, Ivan and Rasoulzadeh, Arvin and Tervooren, Jonas},
title = {Isometric deformations of discrete and smooth T-surfaces},
year = {2024},
issue_date = {2024},
publisher = {Elsevier Science Publishers B. V.},
address = {NLD},
volume = {122},
number = {C},
issn = {0925-7721},
doi = {10.1016/j.comgeo.2024.102104},
journal = {Computational Geometry: Theory and Applications},
numpages = {29}
}

@article{steffen1978symmetric,
  title={A symmetric flexible Connelly sphere with only nine vertices},
  author={Steffen, Klaus},
  journal={preprint IHES, Buressur-Yvette},
  year={1978}
}

@inproceedings{lijingjiao2015asimoptimizing,
  title={Optimizing the Steffen flexible polyhedron},
  author={Lijingjiao, Iila and Tachi, Tomohiro and Guest, Simon D},
  booktitle={Proceedings of IASS Annual Symposia},
  pages={1--10},
  year={2015},
  organization={International Association for Shell and Spatial Structures (IASS)}
}

@article{gallet2024pentagonal,
  title={Pentagonal bipyramids lead to the smallest flexible embedded polyhedron},
  author={Gallet, Matteo and Grasegger, Georg and Legersk{\`y}, Jan and Schicho, Josef},
  journal={arXiv preprint arXiv:2410.13811},
  year={2024}
}

@article{asimow1979rigidity,
  title={The Rigidity of Graphs {II}},
  author={Asimow, Leonard and Roth, Ben},
  journal={Journal of Mathematical Analysis and Applications},
  volume={68},
  number={1},
  pages={171--190},
  year={1979},
  publisher={Elsevier}
}

@article{whiteley1996some,
  title={Some matroids from discrete applied geometry},
  author={Whiteley, Walter},
  journal={Contemporary Mathematics},
  volume={197},
  pages={171--312},
  year={1996}
}

@book{hartshorne2013algebraic,
  title={Algebraic geometry},
    series = {Graduate Texts in Mathematics},
      author={Hartshorne, Robin},
  volume={52},
  year={2013},
  publisher={Springer Science \& Business Media},
    doi = {10.1007/978-1-4757-3849-0}
}

@book{cox2005ideals,
  title={Using Algebraic Geometry},
  author={Cox, David and Little, John and O'shea, Donal},
  year={2005},
  publisher={Springer},
series = {Graduate Texts in Mathematics},
edition = {2nd},
volume={185}
}
